\documentclass[11pt,reqno]{amsart}
\usepackage[a4paper]{geometry}
\usepackage{amsthm,hyperref,lmodern,mathrsfs,amssymb,amsmath,amsfonts}
\usepackage{enumerate}
\usepackage[utf8]{inputenc}
\usepackage{versions}
\usepackage{ulem}

\usepackage{color}
\newcommand\blue[1]{\textcolor{blue}{#1}}
\newcommand\red[1]{\textcolor{red}{#1}}

\allowdisplaybreaks

\numberwithin{equation}{section}

\newcommand{\Hm}[1]{\leavevmode{\marginpar{\tiny%
$\hbox to 0mm{\hspace*{-0.5mm}$\leftarrow$\hss}%
\vcenter{\vrule depth 0.1mm height 0.1mm width \the\marginparwidth}%
\hbox to 0mm{\hss$\rightarrow$\hspace*{-0.5mm}}$\\\relax\raggedright
#1}}}

\newtheorem{theo}{Theorem}[section]
\newtheorem{lem}[theo]{Lemma}
\newtheorem{prop}[theo]{Proposition}
\newtheorem{corol}[theo]{Corollary}

\def\J{{\mathcal{J}\ \!\!}}

\def\R{{\mathbb{R}}}
\def\L{{\mathcal{L}\ \!\!}} 

\def\B{{\mathcal{B}\ \!\!}}

\def\S{{\mathcal{S}\ \!\!}}

\def\C{{\mathcal{C}\ \!\!}}
\def\M{{\mathcal{M}\ \!\!}}

\def\Var{{\mathrm{{\rm Var}}}}

\def\vol{{\mathrm{{\rm vol}}}}

\def\J{{\mathrm{{\rm J}}}}
\def\diam{{\mathrm{{\rm diam}}}}

\def\Int{{\mathrm{{\rm int}}}}
\def\diag{{\mathrm{{\rm diag}}}}

\def\and{{\mathrm{{\rm and}}}}

\def\Jac{{\mathrm{{\rm Jac}}}}

\def\O{{\Omega\ \!\!}}
\def\dO{{\partial \Omega\ \!\!}}

\def\ve{{\varepsilon\ \!\!}}

\def\tmax{{\textrm{max}}}
\def\tmin{{\textrm{min}}}

\def\beq{\begin{equation}}
\def\eeq{\end{equation}}

\begin{document}

\title[Spectral gap on convex bodies]{A note on the spectral gap for log-concave probability measures on convex bodies}

\author{Michel~Bonnefont} \address[M.~Bonnefont]{UMR CNRS 5251, Institut de Math\'ematiques de Bordeaux, Universit\'e Bordeaux , France} \thanks{MB is partially supported by the ANR-23-CE40-0003 Conviviality project and the ANR-19-CE40-0010 QuAMProcs project of the French National Research Agency} \email{\url{mailto:michel.bonnefont(at)math.u-bordeaux.fr}} \urladdr{\url{http://www.math.u-bordeaux.fr/~mibonnef/}}

\author{Ald\'eric~Joulin} \address[A.~Joulin]{UMR CNRS 5219, Institut de Math\'ematiques de Toulouse, Universit\'e de Toulouse, France} \thanks{AJ is partially supported by the ANR-23-CE40-0003 Conviviality project of the French National Research Agency and the ANR LabEx CIMI (grant ANR-11-LABX-0040) within the French  State
Programme ``Investissements d'Avenir".} \email{\url{mailto:ajoulin(at)insa-toulouse.fr}} \urladdr{\url{http://perso.math.univ-toulouse.fr/joulin/}}



\maketitle

\begin{abstract} 
In this paper, we provide explicit lower bounds with respect to some quantities of interest (parameters of the underlying distribution, dimension, geometrical characteristics of the domain, position of the origin, etc.) on the spectral gap of log-concave probability measures on convex bodies. Our results are illustrated by some classical and less classical examples.
\end{abstract}

\section{Introduction}

On a (connected) compact set $\Omega \subset \R^d$ ($d\geq 2$) with smooth boundary $\partial \Omega$ and outer unit-normal $\eta$, we consider a probability measure $\mu$ whose Lebesgue density is proportional to $e^{-V}$, where $V :\Omega \to \R$ is some smooth potential on $\Omega$. One can associate a canonical weighted Laplacian operator $L = \Delta - \langle \nabla V , \nabla \rangle$ endowed with Neumann conditions at the boundary. Under some reasonable assumptions on $V$, it is well-known that the underlying Markov process reflected at the boundary converges in distribution to the invariant and reversible probability measure $\mu$ and the speed of convergence in $L^2 (\mu)$ is governed by the so-called spectral gap $\lambda_1(\Omega, \mu) $ of the operator $-L$, that is, its first positive eigenvalue. In theory it is quite hard to find explicitly the spectral gap beyond product spaces, a situation for which the problem is reduced to the one-dimensional case. Even in this 1D setting, only few examples of explicit constants are known, cf. \cite{barthe_roustant} by means of the Sturm-Liouville theory. To our knowledge, for instance for the uniform distribution in higher dimension, the spectral gap is known explicitly only on Euclidean balls \cite{weinberger} or on some specific triangles \cite{nazarov}. Hence, providing (lower) bounds on the spectral gap that depend conveniently on the dimension is a challenging question that attracted a lot of attention in the last decades, culminating in the famous KLS conjecture. Introduced initially in an isoperimetric context by Kannan, Lov\'asz and Simonovits \cite{kls}, it states equivalently that the spectral gap of the operator $-L$ associated to a convex potential $V$ is of order the inverse of the operator norm of the covariance matrix of $\mu$, cf. for instance \cite{KLS_livre} for a nice introduction to the topic. Actually, the conjecture is almost solved in the sense that, after a series of improvements by several authors using Eldan's localization method, the best and last result is Klartag's one \cite{klartag_KLS} which confirms the conjecture up to some logarithmic prefactor of the dimension. \smallskip

In spirit, the present work differs a bit from this research around the KLS conjecture. Indeed our main motivation is to offer fully explicit theoretical guarantees on the spectral gap which may be useful for practitioners, that is, involving explicit bounds with res-pect to some parameters of interest (parameters of the laws, dimension, geometrical characteristics of $\Omega$, position of the origin, etc.) and not only estimates available up to universal constants. This may be used for example in the Global Sensitivity Analysis of numerical model outputs, a topic which is by now very popular in Statistics and engineering. In these models the most important input variables, which follow some standard distributions but truncated on compact domains, may be determined through some $L^2$ sensitivity indices relating the variance and the energy of the costly computer code function, emphasizing the role of Poincar\'e type inequalities $-$ and thus the spectral gap $-$ in the analysis. We refer to \cite{barthe_roustant} for this approach with one-dimensional input independent variables, together with historical references and credit. 

\smallskip

In order to give an idea of the results we are able to obtain, let us already state one of our main contribution of the paper. Below $\L$ stands for the diagonal matrix operator acting on smooth vector fields $F$ as $\L F = (L F_i)_{i=1,\ldots, d}$, $\J$ stands for the Jacobian matrix and $\rho(A)$ denotes the smallest eigenvalue of a given symmetric matrix $A$. We refer to the next sections for other missing definitions.
\begin{theo}
\label{theo:main}
On a (connected) compact set $\Omega \subset \R^d$ ($d\geq 2$) with smooth boundary $\partial \Omega$ and outer unit-normal $\eta$, we consider a probability measure $\mu$ whose Lebesgue density is proportional to $e^{-V}$, where $V :\Omega \to \R$ is some sufficiently smooth potential on $\Omega$. Let $W$ be some smooth invertible diagonal matrix mapping satisfying the two following assumptions:
\begin{enumerate}
\item[$(A_1)$] The symmetric matrix mapping $\nabla^2 V - \L W \, W^{-1}$ is positive definite on $\Omega$.
\item[$(A_2)$] At the boundary $\dO$ the symmetric matrix mapping $\J \eta - W \, \langle \nabla W^{-1} , \eta \rangle $ (acting as a quadratic form on the tangent space) is non-negative.
\end{enumerate}
Then the generalized Brascamp-Lieb inequality holds: for all $ g \in \C ^\infty (\Omega)$,
$$
\Var _\mu (g) \leq \int_\Omega \left \langle \nabla g , \, (\nabla^2 V - \L W \, W^{-1})^{-1} \nabla g \right \rangle \, d\mu .
$$
In particular the spectral gap $\lambda_1 (\O,\mu)$ satisfies
$$
\lambda_1 (\O,\mu) \geq \inf_{x\in \O} \rho \left( \nabla^2 V (x) - \L W (x) \, W^{-1} (x) \right).
$$
\end{theo}

The proof, delayed in Section \ref{sect:proof_thm}, is based on the intertwining approach we introduced and studied in \cite{ABJ} to obtain Brascamp-Lieb type inequalities on the whole Euclidean space, in the spirit of the famous Brascamp-Lieb inequality established in \cite{brascamp_lieb}. As such, the present study generalizes these results to measures restricted to domains, for which boundary terms have to be considered. Theoretically, our result covers many different settings as soon as we are able to find some $W$ satisfying the announced assumptions. In particular we will see that our approach is relevant when dealing with log-concave probability measures $\mu$ on a convex body $\O$. On the one hand, it includes the uniform distribution on $\O$ for which our estimates improve two completely explicit results under some convenient assumptions: Payne-Weinberger's one \cite{payne_weinberger} involving the diameter of $\O$ and also Klartag's result \cite{klartag} for unconditional convex bodies relying on a kind of monotonicity property. On the other hand, in the presence of a convex potential $V$, we reinforce the Brascamp-Lieb inequality and improve the usual bound on the spectral gap provided by the Bakry-Emery criterion when $V$ is uniformly convex on $\O$. Such an analysis will be illustrated in the case of the Subbotin distribution, which is a radial log-concave probability measure whose potential is not uniformly convex. 
As a final remark, we mention that the applications of Theorem \ref{theo:main} only address the convex setting (convex domains and convex potentials) but we are fully convinced that the method developed here may be applied to some specific non convex situations. This will be the matter of a future work. \smallskip

Let us briefly describe the content of the paper. In Section \ref{sect:material} we recall some basic material and establish preliminary results about Brascamp-Lieb type inequalities and spectral estimates, contained in Lemmas \ref{lemme:K} and \ref{lemme:decompo}, which are at the basis of the proof of Theorem \ref{theo:main} given thereafter. Section \ref{sect:body} is then devoted to apply Theorem \ref{theo:main} when $\Omega$ is a convex body, covering various interesting situations for log-concave distributions (uniform, radial, etc.). It leads to Corollary  \ref{corol:main2bis} (and its consequences) and to Corollary  \ref{corol:main3} in the case of generalized Orlicz balls. Finally in the Appendix, we provide some additional elements on the spectral gap of the uniform distribution on Euclidean balls and discuss in this context a possible optimality of Theorem \ref{theo:main}.

\section{Preliminary results and proof of Theorem \ref{theo:main}}
\label{sect:material}
\subsection{Basic material and notation}

In this paper, we consider on the Euclidean space $(\R^d , \vert \, \cdot \, \vert )$ of dimension $d\geq 2$ a (connected) compact set $\Omega$ with sufficiently smooth boundary (say $\C^2$) $\partial \Omega$ and outer unit-normal $\eta$. Let $\C^\infty (\Omega)$ be the space of infinitely differentiable real-valued functions on $\Omega$.
We introduce a probability measure $\mu$ on $\Omega$ whose Lebesgue density is proportional to $e^{-V}$, where $V : \Omega \to \R$ is some sufficiently smooth potential on $\Omega$, and consider on $\C^\infty (\Omega) $ the associated second-order differential operator
$$
L f = \Delta f - \langle \nabla V , \nabla f \rangle,
$$
endowed with Neumann boundary conditions, \textit{i.e.},
$$
\langle \nabla f , \eta \rangle = 0 \quad \mbox{on} \quad \partial \Omega .
$$
In the sequel we denote $\C_N ^\infty (\Omega) $ ($N$ for Neumann) such a subspace of $\C^\infty (\Omega) $. Above $\Delta$ and $\nabla$ stand respectively for the Euclidean Laplacian and gradient and $\langle \cdot, \cdot \rangle $ is the scalar product. By integration by parts, we have for all $f,g \in \C ^\infty (\Omega)$,
\begin{eqnarray*}
\int_{\Omega} L f \, g \, d\mu & = & \int_\dO g \, \langle \nabla f , \eta \rangle \, d\mu - \int_\Omega \langle \nabla f , \nabla g \rangle \, d\mu \\
& = & \int_\dO \left( g\, \langle \nabla f , \eta \rangle - f\, \langle \nabla g , \eta \rangle \right) \, d\mu + \int_\Omega f \, Lg \, d\mu ,
\end{eqnarray*}
where by abuse of notation we still denote $\mu$ the measure on the boundary with density proportional to $e^{-V}$ with respect to the volume measure on $\dO$. Hence $L$ is symmetric and non-positive on $\C _N ^\infty (\Omega)$ and by completeness it admits a unique self-adjoint extension (still denoted $L$). In particular the (Neumann) spectrum $\sigma (-L)$ of the non-negative operator $-L$ is included in $[0,\infty )$, the zero eigenvalue corresponding to the constant eigenfunctions, and the first positive eigenvalue $\lambda_1 (\O, \mu)$ (denoted as such to emphasize the roles of the domain $\O$ and the probability measure $\mu$), called the spectral gap, is nothing but the optimal constant in the famous Poincar\'e inequality, that is, for all $g\in \C ^\infty (\O)$,
$$
\lambda_1 (\O,\mu) \, \Var_\mu (g) \leq \int_{\Omega} \vert \nabla g \vert ^2 \, d\mu,
$$
where $\Var_\mu (g)$ is the variance of function $g$ under $\mu$,
$$
\Var_\mu (g) = \int_\Omega \left( g - \int_\Omega g \, d\mu \right) ^2 \, d\mu.
$$
Note that the Neumann boundary conditions do not appear directly in the Poincar\'e inequality. \smallskip

Before turning to our first results, let us introduce some notation and definitions. By a matrix mapping (resp. an invertible matrix mapping, resp. a symmetric positive-definite matrix mapping) we mean a map defined on $\O$ and valued in $\mathcal{M}_d (\R)$, the space of $d\times d$ matrices with real entries (resp. in the subset of invertible matrices, resp. in the subset of symmetric positive-definite matrices). Given a smooth matrix mapping $M$ and a smooth vector field $F$ defined on $\Omega$, let $\nabla M$ and $\nabla F$ be respectively the matrix of gradients $(\nabla M_{i,j})_{i,j=1,\ldots,d}$ and the column vector of gradients $(\nabla F_i)_{i=1,\ldots,d}$. If $v \in \R^d$ then we define $\langle \nabla M , v \rangle$ and $\langle \nabla F , v \rangle$ to be respectively the matrix $(\langle \nabla M_{i,j} , v \rangle)_{i,j=1,\ldots,d}$ and the vector $(\langle \nabla F_{i} , v \rangle)_{i=1,\ldots,d}$. Moreover we define the vector field $\nabla M \, \nabla F$ by contraction as
$$
(\nabla M \, \nabla F)_i = \sum_{j=1} ^d \langle\nabla M_{i,j} , \nabla F_j \rangle.
$$
For two column vectors of gradients $\nabla F$ and $\nabla G$ and a symmetric matrix $M \in \mathcal{M}_d(\R)$, we define
$$
\left[ \nabla F\right] ^T M \, \nabla G = \sum_{i,j=1} ^d \langle\nabla F_i , M_{i,j} \nabla G_j \rangle.
$$
Above the superscript $T$ stands for the transpose of a vector or a matrix. Finally we denote $\rho(A)$ the smallest eigenvalue of a given symmetric matrix $A$ and say that $A$ is bounded from below by some constant $\kappa \in \R$ if $\rho(A) \geq \kappa$. If $\kappa=0$ we say that $A$ is non-negative.

\subsection{Brascamp-Lieb type inequalities and general spectral gap estimates}
\label{sect:main}

We start our analysis by stating an important lemma, which is more or less classical at least on the whole Euclidean space $\R^d$, and which might be seen as a dualized Brascamp-Lieb type inequality. Let us give the proof for completeness.
\begin{lem}
\label{lemme:K}
Let $\Omega \subset \R^d$ be a (connected) compact set with smooth boundary $\partial \Omega$ and outer unit-normal $\eta$. Assume that there exists some symmetric positive-definite matrix mapping $K$ such that for every $f \in \C ^\infty _N (\Omega)$,
$$
\int_\Omega (-Lf)^2 \, d\mu \geq \int_\Omega \left \langle \nabla f , K \nabla f \right \rangle \, d\mu .
$$
Then for every $g \in \C ^\infty (\Omega)$, we have the Brascamp-Lieb type inequality
\[
\Var_\mu(g) \leq \int_\O \left \langle \nabla g , K^{-1} \nabla g \right \rangle \, d\mu.
\]
In particular if the mapping $K$ is bounded from below (uniformly with respect to the space variable) by some $\kappa >0$ then the spectral gap of the operator $-L$ is lower bounded as follows:
$$
\lambda_1 (\O,\mu) \geq \kappa .
$$
\end{lem}
\begin{proof}
Letting $g \in \C ^\infty (\Omega)$ be centered, standard results for Neumann type Laplacians ensure the existence of a unique solution $f\in \C _N ^\infty (\Omega)$ to the Poisson equation $-Lf =g$. Then the trick is to write the variance as follows:
\begin{eqnarray*}
\Var_\mu(g) & = & 2 \, \int_\Omega g^2 \, d\mu - \int_\Omega g^2 \, d\mu \\
& = & 2 \, \int_\O  g \, (-Lf) \, d\mu - \int_\O (-Lf)^2 \, d\mu \\
& = & 2 \, \int_\O \left \langle  \nabla g , \nabla f \right \rangle \, d\mu - \int_\O (-Lf)^2 \, d\mu \\
& \leq & 2 \, \int_\O \langle  K^{-\frac{1}{2}} \nabla g , K^{\frac{1}{2}} \nabla f \rangle \, d\mu - \int_\O \langle \nabla f , K \nabla f \rangle \, d\mu \\
& = & \int_\O \langle \nabla g , K^{-1} \nabla g \rangle \, d\mu - \int_\O \vert K^{\frac{1}{2}} \nabla f - K^{-\frac{1}{2}} \nabla g \vert ^2 \, d\mu \\
& \leq & \int_\O \langle \nabla g , K^{-1} \nabla g \rangle \, d\mu .
\end{eqnarray*}
The proof of the spectral gap estimate is then straightforward.
\end{proof}

This approach, known to specialists as the $L^2$ method, is reminiscent of H\"ormander's work \cite{hormander} in the middle of the 60's for solving the Poisson equation associated to the operator $\bar{\partial}$ in complex analysis, and has been used then by several authors to establish Poincar\'e type inequalities. For instance we have in mind the famous (integrated version of the) $\Gamma_2$ curvature dimension criterion of Bakry and Emery \cite{bakry_emery} and also the work of Helffer \cite{helffer} for models arising in statistical mechanics. Moreover Klartag \cite{klartag} used this method to prove, among other things, the variance conjecture in the case of log-concave unconditional distributions, that is, having log-concave density which is invariant under coordinate hyperplane reflections. Recently, he refined the method to prove the KLS conjecture up top some logarithmic factor of the dimension, cf. \cite{klartag_KLS}. The presence of a weight in the inequalities above through the matrix mapping $K$, which is just a refinement of this approach, already appeared for instance in \cite{barthe_cordero, milman_koles} and in our previous papers \cite{ABJ,BJ}, both works aiming at estimating conveniently the spectral gap or higher eigenvalues in various situations of interest. \smallskip

Actually, considering matrix weights takes roots in the pioneer work \cite{brascamp_lieb} through the so-called Brascamp-Lieb inequality for strictly convex potentials $V$. In view of Lemma~\ref{lemme:K}, it corresponds to the mapping $K$ being $\nabla^2 V$, the Hessian matrix of $V$. 
Since we work on the domain $\Omega$, some extra boundary terms involving the second fundamental form appear when integrating the famous Bochner formula adapted to the measure $\mu$, cf. for instance \cite{klartag}. As we will see later, the boundary term is related to the geometry of the domain $\Omega$ in the sense that it reveals to be non-negative when $\O$ is convex. As such, the formula is the following (we omit the proof since it is included in the one of Lemma \ref{lemme:decompo} below): for all $f\in \C _N ^\infty (\Omega)$,
\begin{eqnarray}
\label{eq:classical}
\nonumber \int_\Omega (-Lf)^2 \, d\mu & = & \int_\Omega \Vert \nabla^2 f \Vert _{HS} ^2 \, d\mu + \int_\Omega \langle \nabla f , \nabla^2 V  \nabla f \rangle \, d\mu \\
& & + \int_\dO \langle \nabla f , \J \eta \, \nabla f \rangle \, d\mu,
\end{eqnarray}
where $\Vert A \Vert _{HS} = \sqrt{\sum_{i,j=1, \ldots, d} A_{i,j} ^2}$ stands for the Hilbert-Schmidt norm of a given matrix $A \in \mathcal{M}_d (\R)$ and $\J \eta = (\partial _j \eta _i)_{i,j = 1,\ldots, d}$ denotes the Jacobian matrix of $\eta$. The strict convexity of $V$ (\textit{i.e.}, $\nabla^2 V$ is a symmetric positive-definite matrix mapping) entails by Lemma \ref{lemme:K} the famous Brascamp-Lieb inequality on convex domains: for every $g \in \C ^\infty (\Omega)$,
$$
\Var_\mu(g) \leq \int_\O \langle \nabla g , \nabla ^2 V ^{-1} \nabla  g \rangle\, d\mu.
$$
Finally, if moreover $V$ is uniformly convex on $\O$, that is, $\nabla^2 V$ is uniformly bounded from below by some $\kappa >0$, then the spectral gap of the operator $-L $ on the convex domain $\O$ satisfies
\begin{equation}
\label{eq:unif_cvx}
\lambda_1 (\O , \mu) \geq \kappa ,
\end{equation}
which is the Euclidean version of the Bakry-Emery criterion \cite{bakry_emery}. \smallskip 

In order to reinforce this spectral gap estimate or even to obtain a relevant bound beyond this convex situation, our idea is to introduce some matrix weight in the decomposition \eqref{eq:classical}, freeing us from these strong convexity assumptions. This strategy is inspired by our previous works on the intertwinings, cf. \cite{ABJ,BJ}. See also the approach of Wang \cite{wang} in a different context. Here is our key lemma. Below $\L$ stands for the diagonal matrix operator acting on smooth vector fields $F$ as $\L F = (L F_i)_{i=1,\ldots, d}$ and the notation $W$ is used to remind us that it is interpreted as a weight, the unweighted version (\textit{i.e.}, $W$ is the identity) of our second identity \eqref{eq:weighted} corresponding to the classical decomposition \eqref{eq:classical}.
\begin{lem}
\label{lemme:decompo}
Let $\Omega \subset \R^d$ be a (connected) compact set with smooth boundary $\partial \Omega$ and outer unit-normal $\eta$. Let $W$ be some smooth invertible matrix mapping. Then for all $f \in \C ^\infty (\Omega)$, it holds
\begin{eqnarray*}
\int_\O (-Lf)^2 \, d\mu & = & \int_\O \left[ \nabla (W^{-1} \nabla f) \right] ^T W^T W \, \nabla (W^{-1} \nabla f) \, d\mu \\
& & + \int_\O \left \langle W^{-1} \nabla f , ( \nabla W^T \, W - W^T \, \nabla W ) \, \nabla (W^{-1}\nabla f) \right \rangle \, d\mu \\
& & + \int_\O \left \langle \nabla f , (\nabla^2 V - \L W \, W^{-1} ) \, \nabla f \right \rangle\, d\mu + \int_\dO Lf \, \langle \nabla f , \eta \rangle \, d\mu \\ & & - \int_\dO \langle \nabla f , \nabla^2 f \, \eta \rangle \, d\mu -  \int_\dO \left \langle \nabla f , W \, \langle \nabla W^{-1} , \eta \rangle \, \nabla f \right \rangle \, d\mu .
\end{eqnarray*}
Moreover, if $f$ satisfies the Neumann boundary conditions $\langle \nabla f , \eta \rangle = 0$, then
\begin{eqnarray}
\label{eq:weighted}
\nonumber \int_\O (-Lf)^2 \, d\mu & = & \int_\O \left[ \nabla (W^{-1} \nabla f) \right] ^T \, W^T W \, \nabla (W^{-1} \nabla f) \, d\mu \\
\nonumber & & + \int_\O \left \langle W^{-1} \nabla f , ( \nabla W^T \, W - W^T \, \nabla W ) \, \nabla (W^{-1}\nabla f) \right \rangle \, d\mu \\
\nonumber & & + \int_\O \left \langle \nabla f , (\nabla^2 V - \L W \, W^{-1} ) \, \nabla f \right \rangle\, d\mu \\
& & + \int_\dO \left \langle \nabla f , ( \J \eta - W \, \langle \nabla W^{-1} , \eta \rangle ) \, \nabla f \right \rangle \, d\mu .
\end{eqnarray}
\end{lem}
\begin{proof}
We use the notation $A = W^{-1}$ and $S = (A A^{T})^{-1}$. Recall first the intertwining between operators and (weighted) gradients introduced and studied in \cite{ABJ}:
\begin{equation}
\label{eq:intert}
A\nabla L f = (\L_A -\M_A) \, (A\nabla f),
\end{equation}
where $\L_A$ denotes the matrix operator acting on smooth vector fields as
$$
\L _A F =  \L F + 2 \, A \, \nabla A^{-1} \, \nabla F ,
$$
and $\M_A$ is the matrix corresponding to the multiplicative (or zero-order) operator
$$
\M_A = A \, \nabla ^2 V \, A^{-1} - A \, \L A^{-1} .
$$
We have by integration by parts and the intertwining identity above,
\begin{eqnarray*}
\int_\O (-Lf)^2 \, d\mu & = & \int_\O \langle \nabla f , \nabla (-Lf) \rangle \, d\mu + \int_\dO Lf \, \langle \nabla f ,\eta \rangle \, d\mu \\
& = & \int_\O \langle A \nabla f , S A \nabla (-Lf) \rangle\, d\mu + \int_\dO Lf \, \langle \nabla f ,\eta \rangle \, d\mu \\
& = & \int_\O \left \langle A \nabla f , S (-\L_A + \M_A ) (A\nabla f) \right \rangle \, d\mu + \int_\dO Lf \, \langle \nabla f ,\eta \rangle \, d\mu \\
& = & \int_\O \langle A \nabla f , S (-\L)(A\nabla f) \rangle \, d\mu - \int_\O \langle A \nabla f , 2 \, S A \, \nabla A^{-1} \, \nabla (A\nabla f) \rangle \, d\mu \\
& & + \int_\O \left \langle \nabla f , ( \nabla^2 V - \L A^{-1} \, A ) \, \nabla f \right \rangle \, d\mu + \int_\dO Lf \, \langle \nabla f , \eta \rangle \, d\mu ,
\end{eqnarray*}
since
$$
A^T \, S \, \M_A \, A = \nabla^2 V - \L A^{-1} \, A .
$$
Dealing with the first term in the right-hand-side above, a second integration by parts gives
\begin{eqnarray*}
\int_\O \langle A \nabla f , S (-\L)(A\nabla f) \rangle \, d\mu & = & \int_\O \left[ \nabla (A \nabla f) \right] ^T S \, \nabla (A \nabla f) \, d\mu + \int_\O \left \langle A \nabla f , \nabla  S \, \nabla(A \nabla f) \right \rangle \, d\mu \\
& & - \int_\dO \left \langle S A\nabla f , \langle \nabla ( A \nabla f) , \eta \rangle \right \rangle \, d\mu ,
\end{eqnarray*}
so that reorganizing the terms in the initial computations lead to
\begin{eqnarray*}
\int_\O (-Lf)^2 \, d\mu & = & \int_\O \left[ \nabla (A \nabla f) \right] ^T S \, \nabla (A \nabla f) \, d\mu + \int_\O \left \langle \nabla f , ( \nabla^2 V - \L A^{-1} \, A ) \, \nabla f \right \rangle \, d\mu \\
& & + \int_\O \left \langle A \nabla f , (\nabla  S - 2 \, S A \, \nabla A^{-1}) \, \nabla(A \nabla f) \right \rangle \, d\mu \\
& & - \int_\dO \left \langle S A\nabla f , \langle \nabla ( A \nabla f) , \eta \rangle \right \rangle \, d\mu + \int_\dO Lf \, \langle \nabla f , \eta \rangle \, d\mu \\
& = & \int_\O \left[ \nabla (A \nabla f) \right] ^T S \, \nabla (A \nabla f) \, d\mu + + \int_\O \left \langle \nabla f , ( \nabla^2 V - \L A^{-1} \, A ) \, \nabla f \right \rangle \, d\mu \\
& & + \int_\O \left \langle A \nabla f , \left( (\nabla  A^{-1})^T \, A^{-1} - ( A^{-1})^T \, \nabla A^{-1} \right) \, \nabla(A \nabla f) \right \rangle \, d\mu \\
& & - \int_\dO \langle \nabla f , \nabla^2 f \, \eta \rangle \, d\mu - \int_\dO \left \langle \nabla f , A^{-1} \, \langle \nabla A , \eta \rangle \, \nabla f \right \rangle \, d\mu \\
& & + \int_\dO Lf \, \langle \nabla f , \eta \rangle \, d\mu ,
\end{eqnarray*}
since we have
\begin{eqnarray*}
\int_\dO \left \langle S A\nabla f , \langle \nabla ( A \nabla f) , \eta \rangle \right \rangle \, d\mu & = & \int_\dO \left \langle \nabla f , A^{-1} \, \langle \nabla ( A \nabla f) , \eta \rangle \right \rangle \, d\mu \\
& = & \int_\dO \langle \nabla f , \nabla^2 f \, \eta \rangle \, d\mu + \int_\dO \left \langle \nabla f , A^{-1} \, \langle \nabla A , \eta \rangle \, \nabla f \right \rangle \, d\mu .
\end{eqnarray*}
Hence the first desired identity is proved. Finally under the Neumann boundary conditions $\langle \nabla f , \eta \rangle = 0$ we have
$$
0 = \nabla \langle \nabla f , \eta \rangle = \nabla^2 f \, \eta + (\J \eta )^T \, \nabla f,
$$
from which the announced result follows.
\end{proof}

\subsection{Proof of Theorem \ref{theo:main}}
\label{sect:proof_thm}
Now we are able to prove the first main result of our paper stated in the Introduction, Theorem \ref{theo:main}.
\begin{proof}[Proof of Theorem \ref{theo:main}]
Our aim is to use Lemma \ref{lemme:decompo} to find some convenient matrix mapping $K$ such that Lemma \ref{lemme:K} applies. To do so, we need to understand the four terms arising in the right-hand-side of \eqref{eq:weighted}. The first term is non-negative whereas the second one vanishes since $\nabla W^T \, W = W^T \, \nabla W$, the matrix weight $W$ being diagonal. The most important terms are the two last ones, for which the assumptions of Theorem \ref{theo:main} directly apply: since assumption $(A_2)$ means that the boundary term is non-negative, assumption $(A_1)$ allows us to choose the matrix mapping $K$ equal to $\nabla^2 V - \L W \, W^{-1}$.
\end{proof}



As we will see on the examples, choosing a convenient matrix weight $W$ which ensures simultaneously the conditions $(A_1)$ and $(A_2)$ leading to the desired spectral gap estimate is not an easy task since these conditions are not of the same nature \textit{a priori}. Indeed assumption $(A_1)$, which already appeared in our previous study \cite{ABJ} on the whole space and provides the desired spectral gap estimate, strongly depends on the dynamics through the presence of the Hessian matrix of $V$ and the matrix operator $\L$ whereas $(A_2)$ does not depend on $V$ but only on the geometry of the boundary of the domain.
Therefore the strategy in the sequel is to find some convenient diagonal weight $W$ balancing these two conditions $(A_1)$ and $(A_2)$.


\section{A class of convex bodies}
\label{sect:body}

We concentrate in this part on smooth convex bodies, that is, compact, convex sets of $\R^d$ with non-empty interior and smooth boundary. Before turning to the applications to this convex framework of our main result Theorem \ref{theo:main}, let us briefly recall some elements of the literature in this context. In view of the famous KLS conjecture, many quantitative spectral gap estimates have been established in this convex setting by several authors during the last three decades, using various methods (isoperimetry, optimal transport, etc.). For instance among those important papers we may cite the pioneer work of Kannan, Lov\'asz and Simonovits \cite{kls}, Bobkov's paper  \cite{bobkov_AOP}, or Milman's article \cite{milman}. In particular, an approach which revealed to be efficient is the so-called localization method (or needle decomposition) which have been originally formalized by Payne and Weinberger \cite{payne_weinberger} in the 60s and further developed later in the 80s by Gromov and Milman \cite{gromov_milman} and also in \cite{kls}, allowing the authors to reduce the multidimensional case to a one-dimensional problem. For instance the (log-concave version of the) famous Payne-Weinberger inequality, whose proof relies on this method, is the following: if the potential $V$ is convex on the smooth convex body $\O$, then
\begin{equation} 
\label{eq:Payne_Wein}
\lambda_1 (\O , \mu ) \geq \frac{\pi ^2}{\diam (\O ) ^2},
\end{equation}
where $\diam (\O )= \sup_{x,y \in \O} \vert x-y\vert $ stands for the diameter of $\O$. Although optimal with respect to the diameter by looking at the uniform distribution on one-dimensional intervals, such an estimate still leaves room for improvement in higher dimension since by tensorization the spectral gap for product measures is independent from the dimension 
whereas Payne-Weinberger's inequality \eqref{eq:Payne_Wein} applied for instance to the uniform distribution on the hypercube exhibits a bound of order $1/d$. \smallskip 

In this part we apply Theorem \ref{theo:main} for log-concave measures on smooth convex bodies and improve Payne-Weinberger's estimate under various assumptions. These results correspond to the two other main results of the paper, Corollaries \ref{corol:main2bis} and \ref{corol:main3}. Let us start by introducing the context in which our study takes place. In most of the cases of interest the domain $\O$ we consider is of the form
$$
\Omega = \{ x\in \R^d : F(x) \leq 0 \} ,
$$
where $F$ is some smooth convex function defined on $\R^d$. Then the boundary is described by the algebraic equation $F = 0$, the outer unit-normal is given by $\eta = \nabla F / \vert \nabla F \vert$ provided the gradient does not vanish at the boundary. The second fundamental form of the boundary is defined as the following quadratic form: for all vectors $u,v \in T_x \O = \{ m\in \R^d : \langle m, \eta (x) \rangle = 0 \}$, the tangent space at point $x\in \dO$, 
$$
\langle u , \J \eta (x) \, v \rangle =  
\frac{1}{\vert \nabla F (x) \vert } \, \langle u , \nabla^2 F (x) \, v \rangle ,
$$
so that the convexity of $\O$ (induced by the convexity of $F$) implies that the second fundamental form is non-negative. In the sequel, even if not explicitly written, we will always let the matrix $\J \eta (x)$ (and also the various expressions involving it) act on the tangent space $T_x \O$ only. Moreover, in order to lighten the notation, we will often omit the space variable by writing $\J\eta $. The eigenvalues of the second fundamental form are called the principal curvatures and we denote $\rho (\J\eta )$ the smallest one (depending on the space variable). \smallskip

Three interesting cases we will consider are the following:

$\circ$ The uniform distribution on $\O$, \textit{i.e.}, the normalized volume measure with probability density function $1_\O / \vol (\O)$. The potential $V$ is null and the underlying operator $L$ is the Laplacian $\Delta$ on $\O$. In this context we denote for simplicity the spectral gap $\lambda_1 (\O)$. \smallskip 

$\circ$ A log-concave probability measure related to a uniformly convex potential $V$ on $\O$. Such measures include for instance the standard Gaussian case, \textit{i.e.}, $V = \vert \cdot \vert ^2 /2$.\smallskip 

$\circ$ A radial probability measure on $\O$: the associated potential $V: \O \to \R$ only depends on the Euclidean norm, that is, $V(x) = V(\vert x\vert )$ (using an obvious abuse of notation). In this case $\mu$ is log-concave as soon as the one-dimensional function $V$ is convex and non-decreasing on $\R^+$. A typical example of a radial log-concave probability measure is the Subbotin distribution, \textit{i.e.}, $V = \vert \cdot \vert ^\alpha /\alpha$ with $\alpha \geq 1$. \smallskip


\smallskip
In order to apply Theorem \ref{theo:main}, we will need some relevant assumption on the convex body $\O$. To do so, some preparation is required. Assume first that the origin lies in the interior of the convex body (in the sequel we note $0 \in \Int (\O )$). Let $R_{\max}$ denotes the smallest positive number such that $\O \subset \B (0,R_{\max})$, the (closed) Euclidean ball centered at the origin and with radius $R_{\max}$. Then it is easily seen that $R_{\max}$ is of order of the diameter of $\Omega$. Moreover we have $\langle x, \eta (x) \rangle \geq 0$ for all $x\in \dO$, see for instance Section 1.3 in \cite{schneider}. Actually we even have $\langle x, \eta (x) \rangle \geq R_{\min} $ where $R_{\min} (>0)$ is the largest positive number such that $\B (0,R_{\min}) \subset \O $. Consider the following assumption: there exists some $\beta >0$ such that
\begin{equation}
\label{eq:beta_1}
\J \eta (x) \geq \beta \, \frac{\langle x, \eta (x) \rangle }{r^2} \, I , \quad x\in \dO ,
\end{equation}
where above and in the remainder of the paper we denote $r = \vert x\vert $ to simplify the notation. It is easy to show that this assumption is equivalent to the fact that $\O$ is uniformly convex, \textit{i.e.}, the second fundamental form is uniformly bounded from below by some positive constant (however it is a weaker assumption when the convex domain $\Omega$ is unbounded). Indeed since $r\in [R_{\min }, R_{\max}]$ at the boundary, the optimal constants satisfy 
$$
\inf _{\dO} \rho (\J \eta ) \geq \beta R_{\min }/ R_{\max} ^2 \quad \mbox{ and } \quad \beta \geq \inf _{\dO} \rho (\J\eta ) \, R_{\min}.
$$ 
Actually, the assumption \eqref{eq:beta_1} already appeared as a key hypothesis in the work of Kolesnikov and Milman \cite{KM} about Brascamp-Lieb type inequalities and spectral gap estimates. We will come back to their article in a moment. Note also that the right-hand-side of \eqref{eq:beta_1}, which is well-defined since $r \geq R_{\min} >0$ at the boundary, depends on the position of the origin contrary to the uniform convexity assumption. Dealing now with the key parameter $\beta$, we point out that we always have $\beta \leq 1$. Indeed, let $x_0 \in \dO$ be a point intersecting the boundary and the sphere of radius $R_{\min}$ and centered at the origin. On the one hand the outer unit-normal at $x_0$ is the same for $\Omega$ and $\B(0,R_{\min})$ and is given by $\eta (x_0) = x_0 /R_{\min}$, so that the assumption \eqref{eq:beta_1} gives
$$
\J \eta (x_0) \geq \frac{\beta}{R_{\min}} \, I .
$$
On the other hand at point $x_0$ the second fundamental form is at most that of the ball $\B(0,R_{\min})$, which is $(1/R_{\min }) I$. Combining those two arguments yields the desired conclusion. In particular $\beta =1$ when $\Omega $ is an Euclidean ball centered at the origin. \smallskip

Assume now that $0 \notin \O$ (now and in the remainder of the paper we will not consider the unessential case $0 \in \dO$ since it would require some additional technicalities in the approximation procedures). It leads to some important observations: on the one hand there exist some points  $x\in \dO$ of the boundary such that $\langle x, \eta (x) \rangle < 0$ and on the other hand $R_{\max}$ might be much larger than the diameter of $\O$ when $d(0,\O )$, the distance from $\O$ to the origin, is very large. In particular the right-hand-side in assumption \eqref{eq:beta_1} could be non-positive, meaning that $\O$ would no longer necessarily be convex. To preserve convexity, our idea is to slightly modify \eqref{eq:beta_1} by considering the following assumption: 
\begin{equation}
\label{eq:beta_plus}
\J\eta (x) \geq \beta \, \frac{\langle x, \eta (x) \rangle ^+ }{r^2} \, I , \quad x\in \dO,
\end{equation}
where the $+$ denotes the positive part. 
Such assumption, which coincides with \eqref{eq:beta_1} when $0 \in \Int (\O )$, is the key hypothesis we will use in the sequel. However \eqref{eq:beta_plus} does not necessarily entail that $\beta \leq 1$ and moreover it is not equivalent to the uniform convexity of the domain (it is actually a weaker assumption since the optimal $\beta$ in \eqref{eq:beta_plus} satisfies $\beta \geq \inf_{\dO} \rho (\J \eta ) \, d(0,\O)$).

\subsection{General log-concave probability measures}
\label{sect:general_LC}
We are now in position to state the first important consequence of Theorem \ref{theo:main}. It corresponds to the second main result of our paper.
\begin{corol}
\label{corol:main2bis}
Consider a log-concave probability measure $\mu$ on a smooth convex body $\O \subset \R^d$ ($d>3$). Assume that there exists some $\beta >0$ such that \eqref{eq:beta_plus} is satisfied and
\begin{equation}
\label{eq:beta_V}
\nabla^2 V (x) \geq \beta \, \frac{\langle x ,\nabla V (x) \rangle ^+ }{r^2} \, I, \quad x\in \O.
\end{equation}
Then the generalized Brascamp-Lieb inequality holds: for all $ g \in \C ^\infty (\Omega)$,
\begin{equation}
\label{eq:GBL_r2}
\Var _\mu (g) \leq \frac{1}{\beta (d-2 - \beta )} \, \int_{\Omega} r^2 \, \vert \nabla g \vert ^2 \, d\mu .
\end{equation}
In particular the spectral gap satisfies
$$
\lambda_1 (\O,\mu) \geq \frac{\beta (d-2 - \beta )}{ R_{\max} ^2}.
$$
\end{corol}
\begin{proof}
Since the generalized Brascamp-Lieb inequality implies directly the desired spectral gap estimate, let us prove only \eqref{eq:GBL_r2}. Given $\varepsilon >0$, we choose some weight $W_\varepsilon $ which is a multiple of the identity and radial, say $W_\varepsilon  (x) = w_\varepsilon  (r) \, I$ for all $x\in \Omega$, where $w_\varepsilon (r)  = (\varepsilon + r^2)^{-\beta /2}$. Then the quantities of interest appearing in the assumption $(A_2)$ of Theorem \ref{theo:main} rewrite as follows: at the boundary $x\in \dO$, we have
\begin{align*}
\J \eta (x) - W_\varepsilon (x)  \, \langle \nabla W_\varepsilon ^{-1} (x) , \eta (x) \rangle & = \J\eta (x) + \frac{\langle \nabla w_\varepsilon(r) , \eta (x) \rangle}{w_\varepsilon(r) } \, I \\
& = \J\eta (x) + \frac{w_\varepsilon ' (r) }{w_\varepsilon (r)} \, \frac{\langle x , \eta (x) \rangle}{r}  \, I \\
& = \J\eta (x) - \frac{\beta r }{\ve + r^2} \, \frac{\langle x , \eta (x) \rangle}{r}  \, I \\
& \geq \left( \beta - \frac{\beta r^2 }{\ve + r^2} \right) \, \frac{\langle x , \eta (x) \rangle ^+}{r^2}  \, I. 
\end{align*}
Hence the assumption $(A_2)$ of Theorem \ref{theo:main} is satisfied regardless of the value of $\varepsilon$. Now let us come back to assumption $(A_1)$ of Theorem \ref{theo:main} and see if it is satisfied with this choice of function $w_\varepsilon$. We have for all $x\in \O$, 
\begin{align*}
\frac{\Delta w_\ve (r)}{w_\ve (r)} & = \frac{w_\ve '' (r)}{w_\ve (r)} + \frac{d-1}{r} \, \frac{w_\ve '(r)}{w_\ve (r)} \\
& = \frac{\beta (\beta +1) r^2 - \beta \ve}{(\ve + r^2)^2} - \frac{\beta (d-1)}{\ve + r^2},   
\end{align*}
so that we get 
\begin{eqnarray*}
\nabla^2 V (x) - \L W_\varepsilon(x) \, W_\varepsilon ^{-1} (x) & = & \nabla^2 V (x) + \left( \frac{- \Delta w_\varepsilon (r)}{w_\varepsilon(r)} + \frac{\langle \nabla w_\varepsilon (r) , \nabla V (x) \rangle}{w_\varepsilon (r)} \right) \, I \\
& = & \nabla^2 V (x) + \left( \frac{\beta d \varepsilon + \beta (d-2-\beta ) r^2}{(\varepsilon +r^2)^2} - \frac{\beta \langle x, \nabla V (x) \rangle }{\varepsilon + r^2} \right) \, I \\
& \geq & \left( \frac{\varepsilon \, \beta \, \langle x , \nabla V (x) \rangle ^+}{(\varepsilon + r^2 ) r^2 } + \frac{\beta d \varepsilon + \beta (d-2-\beta ) r^2}{(\varepsilon +r^2)^2} \right) \, I \\
& \geq & \frac{\beta d \varepsilon + \beta (d-2-\beta ) r^2}{(\varepsilon +r^2)^2} \, I,
\end{eqnarray*}
according to the assumption \eqref{eq:beta_V}. Hence by Theorem \ref{theo:main} we obtain the following generalized Brascamp-Lieb inequality: for all $ g \in \C ^\infty (\Omega)$,
$$
\Var _\mu (g) \leq \int_{\Omega} \frac{(\varepsilon + r^2)^2}{\beta d \varepsilon + \beta (d-2-\beta) r^2} \, \vert \nabla g \vert ^2 \, d\mu .
$$
Finally, using the dominated convergence theorem on the compact set $\O$ as $\varepsilon $ tends to 0 yields to the desired generalized Brascamp-Lieb inequality \eqref{eq:GBL_r2}. The proof is now complete.
\end{proof}

Comparing to Payne-Weinberger's estimate \eqref{eq:Payne_Wein}, our result is relevant in most cases except one: when $R_{\max}$ is large compared to the diameter. Such a situation occurs only when $d(0,\O ) $ is very large. We will address this problem later in Section \ref{sect:subbotin} by considering the example of the Subbotin distribution. \smallskip 

As announced, we come back to the work of Kolesnikov and Milman \cite{KM} and in particular their Theorem 6.7 established by a conformal change of (Riemannian) metric combined with Bakry-Emery criterion. Actually, Corollary \ref{corol:main2bis} recovers their result (with slightly better constants) when $0 \in \Int (\O )$ and extends it to more general log-concave probability measures than the uniform distribution. It is tempting to wonder if our intertwining method can be seen as a rewriting of their approach by choosing conveniently in \eqref{eq:intert} the matrix weight $A$. After a careful reading of their paper and some attempts to obtain clear correspondences, it seems to us that both approaches are not equivalent, although our choice of function $w_\varepsilon$ in our proof corresponds to their conformal transformation. The only situation for which both approaches coincide is the case of product metrics, the latter being relevant in the study of unconditional log-concave probability measures. \smallskip

For the second assumption \eqref{eq:beta_V} involving also the potential $V$ on the convex body $\Omega$, Corollary \ref{corol:main2bis} is relevant in the three aforementioned frameworks: 

$(i)$ $V \equiv 0$ and $\mu$ is the uniform distribution on $\O$; 

$(ii)$ $V$ is uniformly convex; 

$(iii)$ $V$ is convex and radial. \smallskip 

\noindent Let us now investigate in detail how this result may be applied to these situations.

\subsection{The uniform distribution}
\label{sect:unif}
Let us start by case $(i)$. Dealing with the uniform distribution, the spectral estimate of Corollary \ref{corol:main2bis} has to be interpreted as a spectral gap comparison between convex bodies and Euclidean balls. Given a radius $R>0$, it is more or less known to specialists (see for instance our previous article \cite{bjm} for a proof based on the underlying radial structure) that $\lambda_1 (\B (0,R))$ is of order $d/R^2$ (recalled in the Appendix, its exact expression does not exhibit an explicit behaviour with respect to the dimension). Actually, Corollary \ref{corol:main2bis} enables to recover easily this estimate since we have $\beta = 1$ and thus we get in dimension $d >3$,
$$
\lambda_1 ( \B (0,R) ) \geq \frac{d-3}{R^2} .
$$
We mention that a similar estimate which is available in dimension 2 and 3 might be obtained by rather considering in the proof of Corollary \ref{corol:main2bis} the radial function $w(r) = \exp \, (- r^2 /2R^2)$, leading to the slightly better estimate $\lambda_1 (\B (0,R)) \geq (d-1)/R^2$. See also the discussion in the Appendix. Now the question is the following: since Weinberger \cite{weinberger} proved the following inequality:
$$
\lambda_1 ( \O ) \leq \left( \frac{\vol (\B (0,1)) }{\vol (\O )} \right) ^{2/d} \, \lambda_1 ( \B (0,1)),
$$
\textit{i.e.}, the spectral gap $\lambda_1 ( \B (0,1)) $ of the Euclidean unit ball $\B (0,1)$ maximizes all the spectral gaps $\lambda_1 (\O )$ of bounded domains $\Omega$ with the same volume, does the reverse Weinberger inequality hold (up to some constants) for a convenient class of domains so that both spectral gaps would be of the same order ? Note that the convexity is a reasonable assumption in order to avoid bottlenecks (and thus arbitrarily small spectral gaps), but it is clearly not sufficient as suggested by the dimension free estimate for the hypercube. Actually, the desired inequality reveals to be true by reformulating the conclusion of Corollary \ref{corol:main2bis} at least under the assumption \eqref{eq:beta_plus}. This observation leads to the following result, which is the announced spectral gap comparison. Since in the uniform case the spectral gap is translation invariant, we can assume that $0 \in \Int ( \O )$ and the position of the origin has then to be optimized to obtain the best possible estimate. 
\begin{prop}
\label{prop:comparison}
Assume that the smooth convex body $\O \subset \R^d$ with $0 \in \Int (\O )$ satisfies the assumption \eqref{eq:beta_plus}. Then we have the spectral gap comparison
$$
\lambda_1 ( \O ) \geq \frac{\beta (d-2-\beta) }{d+2} \left( \frac{R_{\min}}{R_{\max}}\right) ^2 \left( \frac{\vol (\B (0,1)) }{\vol (\O )} \right) ^{2/d} \lambda_1 ( \B (0,1)) .
$$
\end{prop}
\begin{proof}
The proof is straightforward from Corollary \ref{corol:main2bis}. Indeed we have the estimate 
$$
\lambda_1 (\O ) \geq \frac{\beta (d-2 - \beta )}{ R_{\max} ^2}.
$$
Now we have $\B (0, R_{\min}) \subset \O $ leading to the volume comparison
$$
(R_{\min}) ^d \, \vol (\B (0,1)) \leq \vol (\O ) .
$$
Testing on linear functions, the spectral gap $\lambda_1 ( \B (0,1))$ is easily bounded from above as follows:
$$
\lambda_1 ( \B (0,1))  \leq \frac{d}{\int_{\B (0,1)} \vert x\vert ^2 \, dx } = \frac{d \, \int_0 ^1 r^{d-1} \, dr }{\int_0 ^1 r^{d+1} \, dr } = d+2.
$$
Finally combining all these inequalities entails the desired result.
\end{proof}
In particular, this estimate is relevant as soon as $R_{\min}$ and $R_{\max}$ are of the same order, meaning in some sense that the convex body $\O$ is close to an Euclidean ball. We refer to Milman's paper \cite{milman} and in particular Section 5 in which the author obtains several spectral gap comparison results for convex bodies under volume preserving perturbations or in terms of total variation distance. \smallskip

Before turning to the case $(ii)$ of a uniformly convex potential $V$, let us justify at least for the ball why we may assume in Proposition \ref{prop:comparison} that $0 \in \Int (\O )$. Let $\O$ be the Euclidean ball $\B (a,R)$ of radius $R$ and centered at point $a = a_1 e_1$ with $a_1 \geq 0$, where $e_1$ denotes the first vector of the standard canonical basis of $\R^d$. For $x$ at the boundary, we have $\eta (x) = x/R$, $\J\eta (x) = (1/ R) \, I$ and 
\begin{align*}
\frac{\langle x, \eta (x) \rangle ^+}{r^2} 
& = \frac{\langle a_1 e_1 + R \eta(x), \eta(x) \rangle ^+}{\Vert a_1 e_1 + R \eta(x) \Vert^2}\\
& = \frac{\left( a_1 \langle e_1, \eta(x) \rangle + R \right) ^+}{ a_1 ^2 + R^2  + 2 R a_1 \langle e_1,\eta(x) \rangle }.
\end{align*}
Thus
\[
\sup_{x\in \partial \Omega} \frac{\langle x, \eta (x) \rangle ^+}{r^2}  = \sup_{u\in[-1,1]} \frac{\left( a_1 u + R\right) ^+}{ a_1 ^2 + R^2  + 2 R a_1 u } .
\]
If $0 \in \Int (\B (a,R) ) $, that is if $a_1 < R$, then the supremum is attained at $u=-1$ and has value $1/(R-a_1)$ and the condition \eqref{eq:beta_plus} holds if and only if  
\[
\beta \leq 1 - \frac{a_1}{R} \, (\leq 1).
\]
If now $0 \notin \B (a,R) $, that is if $a_1 > R$, then 
the supremum is attained at $u= 1$ and has value $1/(R+ a_1)$ so that the condition \eqref{eq:beta_plus} holds if and only if 
\[
\beta \leq 1+ \frac{a_1}{R}.
\]
Denote $\beta_{\max}$ the above bound on $\beta$ according to the situation we consider. Then the best lower bound on the spectral gap appearing in Corollary \ref{corol:main2bis} is   
\[
\frac{\beta_{\max} (d-2-\beta_{\max})}{R_\tmax^2} = \frac{\beta_{\max} (d-2-\beta_{\max})}{(R+a_1)^2} .
\]
Of course $\beta_{\max}$ has to be $< d-2$ to entail a relevant estimate. In the case $a_1 < R$, we have $\beta_{\max} = 1- a_1 /R$ and the previous bound behaves asymptotically as $(d-3)/R^2$ as $a_1$ is small (it is actually reached for $a_1 = 0$) whereas if $a_1 > R$, $\beta_{\max} = 1+ a_1 /R$ and
\[
\frac{\beta_{\max} (d-2-\beta_{\max})}{(R+a_1)^2} = \frac{d-3 - \frac{a_1}{R}}{R^2 \left( 1+ \frac{a_1}{R} \right)},
\]
which is always $< (d-3)/R^2$. 

\subsection{The uniformly convex case}
\label{sect:unif_convex}
We consider briefly the case $(ii)$ for which the potential $V$ is uniformly convex on $\O$: there exists $\alpha_1 >0$ such that $\nabla^2 V (x) \geq \alpha_1 \, I$ for all $x\in \O$. Assume to simplify that $0 \in \Int (\O )$  and $V$ attains its unique mininum at the origin. By compactness there exists $\alpha_2 >0$ such that $\nabla^2 V (x) \leq \alpha_2 \, I,$ $x\in \O$. Then the assumption \eqref{eq:beta_V} holds for all $\beta \in (0,\alpha_1/\alpha_2]$ hence Corollary \ref{corol:main2bis} applies for all $\beta \in (0,\alpha_1/\alpha_2]$ such that the assumption \eqref{eq:beta_plus} is satisfied. If the convex domain $\Omega$ is not bounded, that is, $R_{\max}$ is infinite, but still satisfies the assumption \eqref{eq:beta_plus}, then the spectral estimate becomes irrelevant but the generalized Brascamp-Lieb inequality \eqref{eq:GBL_r2}  remains available under the assumption that $\alpha _2$ is finite, concerning possibly interesting and non classical situations (we have in mind for instance a two dimensional standard Gaussian distribution restricted to the epigraph of a parabola including the origin). 

\subsection{The Subbotin distribution}
\label{sect:subbotin}
Let us finally consider the radial log-concave case $(iii)$. We have $\nabla V (x) = V'(r) x/r$ and  
$$
\nabla ^2 V (x) = \frac{V'(r)}{r} \, I + \left( V''(r) - \frac{V'(r)}{r} \right) \, \frac{x x^T}{r^2}, \quad x\in \Omega .
$$
The eigenvalues are $V''(r)$ and $V'(r)/r$ with respective eigenspace $\R x$ and $(\R x) ^\perp$, its orthogonal complement. Hence a sufficient condition ensuring the assumption \eqref{eq:beta_V} is the following:  
$$
\min \left \{ V''(r), \frac{V'(r)}{r}\right \} \geq \frac{\beta \, V'(r)}{r}, \quad x\in \O .
$$

The aim of this part is to investigate the particular case of the Subbotin distribution, which is an interesting example worthy to be investigated in detail. Recall that the potential $V$ on the convex body $\O$ is of the form $V = \vert \cdot \vert ^\alpha /\alpha$ for $\alpha > 1$ (actually, the case $\alpha = 1$ could also be considered as well, but would require a slight modification of the argument below). By \eqref{eq:unif_cvx} we have the estimate
\begin{align*}
\lambda_1 \left( \O , \mu \right) & \geq \inf_{x\in \O} \, \min \left \{ V''(r) , \frac{V'(r)}{r}\right \} \\
& = \inf_{x\in \O} \, \min \{ 1, \alpha-1 \} \, r^{\alpha-2} .
\end{align*}
For $\alpha \geq 2$ it yields 
\begin{equation}
\label{eq:BE_Sub2}
\lambda_1 \left( \O , \mu \right) \geq d( 0, \O)^{\alpha-2}.
\end{equation}
In particular if $0 \in \Int (\O )$ then the potential $V$ is not uniformly convex at the origin. For $\alpha \in (1,2]$ it is uniformly convex and we have 
\begin{equation}
\label{eq:BE_Sub}
\lambda_1 \left( \O , \mu \right) \geq (\alpha-1 ) \, R_{\max} ^{\, \alpha-2}.
\end{equation}
Actually, we will see that, except for the standard Gaussian case $\alpha=2$, these bounds are not sufficient in general and have to be reinforced to reach the sharp order of the spectral gap with respect to the dimension. For instance if $\O$ is the ball $\B (0,R)$ for some $R>0$, we expect a competition between two different regimes, depending on the position of $R$ $(= R_{\max})$ with respect to the average value $\int_{\R^d} \vert x \vert \, d\mu (x)$, which is of order $d^{1/\alpha}$: when $R \ll d^{1/\alpha}$ the spectral gap should be comparable to that of the uniform distribution on the ball $\B (0,R)$, which is of order $d/R^2$ as we have seen previously, since in this case the Subbotin distribution is close to the uniform law on $\B (0,R)$ (for instance in total variation distance), whereas for large $R \gg d^{1/\alpha}$ the regime we expect should be similar to that of the Subbotin on the whole space $\R^d$, which is approximatively $d^{1- 2/\alpha}$, cf. \cite{bjm}. In all these cases we observe that the estimate \eqref{eq:BE_Sub} obtained by using the uniform convexity is not sufficient to reach the expected results. \smallskip

Now it is time to state our result which can be seen as a refinement of Corollary \ref{corol:main2bis} in the Subbotin case. Below we allow the parameter $\beta$ given by \eqref{eq:beta_plus} to be null since some of the results are still relevant for general convex bodies (for the other situations it does not bring any information since the lower bound obtained on the spectral gap vanishes). Similarly, we state the result for $d >3$ but some of the statements are relevant for $d=2$ or $d=3$ as long as the lower bound obtained on the spectral gap is positive (although for small dimension our results are somewhat comparable to Payne-Weinberger's inequality \eqref{eq:Payne_Wein} as soon as $\diam (\Omega)$ and $R_\tmax$ are of the same order).
\begin{prop}
\label{prop:subbotin}
Let $\O \subset \R^d$ ($d>3$) be a smooth convex body satisfying \eqref{eq:beta_plus} for some $\beta \geq 0$. Let $\mu$ be the Subbotin distribution with parameter $\alpha >1 $ on $\O$. Denote $\beta_\alpha = \min \{ \beta ,1, \alpha-1 \}$. Then the spectral gap satisfies: 
\begin{itemize}
\item If $\alpha >2$, then we have 
\[
\lambda_1 \left( \O , \mu \right) \geq  \max \left\{ \frac{ \beta_\alpha (d-2- \beta_\alpha)}{R^2_{\max}}, \,
C_{\alpha} \, \beta_\alpha \, \left(d- 2 -\frac{\beta_\alpha}{2}\right)^{1- 2/\alpha} , d( 0, \O)^{\alpha-2} \right\} ,
\]
with
$
C_{\alpha} = \frac{\alpha}{4} \left( \frac{\alpha-2}{2}  \right)^{(2-\alpha)/\alpha}.
$ \smallskip

\item If $\alpha \in (1,2]$, then 
\[
\lambda_1 \left( \O , \mu \right) \geq \max \left \{ \frac{ \beta_\alpha \,  (d+ \alpha -2 - \beta_\alpha ) }{R^2_{\max}} , (\alpha-1) \, R_{\max} ^{\, \alpha-2} \right \}.
\]
In the case $\alpha \in (1,2)$, if moreover $0 \in \Int (\O )$ then the estimate can be improved for large $R_{\max}$ as 
\[
\lambda_1 \left( \O , \mu \right) \geq \frac{\alpha }{4} \left( \frac{2-\alpha} {\alpha-1}\right) ^{1-2/\alpha} \, (d+\alpha-2)^{1-2/\alpha} . 
\]
\end{itemize}
\end{prop}
\begin{proof}
We first consider the case $\alpha >2$. Then we take $W = w I $ with $w$ the radial function $w(r)=r^{-b}$ and $b>0$ to be chosen later (if $0 \in \Int (\O )$ then a regularization procedure similar to that emphasized in the proof of Corollary \ref{corol:main2bis} is required; we omit the details). First as before, the hypothesis $\eqref{eq:beta_plus}$ allows us to bound from below the term at the boundary $x\in \dO$ as follows: 
\begin{align*}
\nonumber \J\eta (x) - W (x)  \, \langle \nabla W ^{-1} (x) , \eta (x) \rangle & = \J \eta (x) + \frac{w ' (r) }{w(r)} \, \frac{\langle x , \eta (x) \rangle}{r}  \, I \\
\label{eq:bord-1-sur-rb} 
& = \J \eta (x) - b \, \frac{\langle x , \eta (x) \rangle}{r^2}  \, I \\
\nonumber & \geq \left( \beta - b \right) \, \frac{\langle x,\eta (x) \rangle ^+}{r^2} \, I.
\end{align*}
Hence the assumption  $(A_2)$ of Theorem \ref{theo:main} is satisfied provided $b \leq \beta$. We now concentrate on the term inside the domain in assumption $(A_1)$. For $x\in \O$, the smallest eigenvalue of the matrix $\nabla^2 V (x) - \L W (x) \, W^{-1} (x)$  is given by
\begin{equation}\label{eq:subb:1-r}
\rho \left( \nabla^2 V (x) - \L W (x) \, W^{-1} (x)\right)  =  (1- b)   \, r^{\alpha-2} + \frac{b (d-2-b)} {r^2} .
\end{equation} 
On the one hand taking $b= \beta_\alpha = \min \{ \beta,1 \}$ yields for all $x\in \O$,
$$
\rho \left( \nabla^2 V (x) - \L W (x) \, W^{-1} (x)\right) \geq \frac{ \beta_\alpha (d-2- \beta_\alpha)}{R^2_{\max}} .
$$
On the other hand when $R_\tmax $ is large, it is better to keep the first term in \eqref{eq:subb:1-r} and to bound from below the whole expression by the minimum on $(0,\infty)$ of the function
\[
g:r  \mapsto (1- b)   \, r^{\alpha-2} + \frac{b (d-2-b)} {r^2}.
\]
It is obtained in 
$$
r_0 = \left( \frac{2 b(d-2-b)} {(\alpha-2) (1-b)} \right)^{1/\alpha},
$$ 
and has value
\[
g(r_0) = \alpha \left( \frac{1-b}{2}  \right)^{2/\alpha} \left( \frac{b(d-2-b)}{\alpha- 2}  \right)^{1 - 2/\alpha}. 
\]
Choosing for simplicity $b = \beta_\alpha /2$ leads to the more presentable lower bound:
$$
g(r_0) \geq \frac{\alpha}{4} \left( \frac{\alpha-2}{2}  \right)^{(2-\alpha)/\alpha} \, \beta_\alpha \, \left(d- 2 -\frac{\beta_\alpha}{2}\right)^{1- 2/\alpha} .
$$
Finally using \eqref{eq:BE_Sub2} completes the proof in the case $\alpha >2$. \smallskip \\
We now turn to the case $1<\alpha \leq 2$ for small $R_{\max}$. Although the same function $w(r) = r^{-b}$ as above may be used, 
the function $w(r) =  \exp \, (- \varepsilon r^\alpha /\alpha )$ with $\varepsilon >0$ chosen below produces a slightly better lower bound. By $\eqref{eq:beta_plus}$, the term at the boundary $x\in \dO$ satisfies
\begin{align*}
\label{eq:bord-exp1} 
\J \eta (x) + \frac{w ' (r) }{w(r)} \, \frac{\langle x , \eta (x) \rangle}{r}  \, I 
& = \J\eta (x) - \ve r^{\alpha-1} \, \frac{\langle x , \eta (x) \rangle}{r}  \, I \\
\nonumber & \geq \frac{\langle x,\eta (x) \rangle ^+}{r^2} \, \left( \beta - \ve r^\alpha \right) \, I,  
\end{align*}
hence the assumption $(A_2)$ of Theorem \ref{theo:main} is satisfied as soon as $\ve \leq \beta / R_\tmax ^{\alpha}$. 
Moreover, since 
\begin{align*}
\quad \frac{\Delta w(r)}{w(r)} & = \frac{w''(r)}{w(r)} + \frac{d-1}{r} \, \frac{w'(r)}{w(r)} \\
& = \ve^2  r^{2(\alpha-1)} - \ve (d+\alpha-2)r^{\alpha-2} , 
\end{align*}
one has  
\begin{align*}
\rho \left( \nabla^2 V (x) - \L W (x) \, W^{-1} (x)\right) 
& = (\alpha-1) r^{\alpha-2} + \ve (d+\alpha-2)r^{\alpha-2} - \ve r^{2(\alpha-1)} - \ve^2  r^{2(\alpha-1)} \\
& =(\alpha-1 - \ve  r^{\alpha}) r^{\alpha-2} + \ve (d +\alpha- 2 - \ve r^{\alpha}) r^{\alpha-2} 
\\
& \geq (\alpha-1 -b) r^{\alpha-2}+ \frac{b  (d+\alpha- 2-b)}{R_{\tmax}^\alpha} r^{\alpha-2}\\
& \geq  \frac{b  (d+\alpha-2-b)}{R_{\tmax}^2} ,
\end{align*}
the two inequalities being consequences of the choice $\ve = b / R_{\tmax}^{\alpha}$ with $b \leq \beta_\alpha = \min \{ \beta,\alpha-1 \}$. As such, choosing finally $b= \beta_\alpha$ yields the desired bound for sufficiently small $R_{\max}$ in the case $\alpha \in (1,2]$. \\
Let us achieve the proof by providing the relevant bound for large $R_{\max}$. As noticed earlier, the standard Gaussian case $\alpha=2$ is straightforward by using the sharp estimate \eqref{eq:BE_Sub}, which also holds when $1 < \alpha < 2$. However it can be improved in the latter case to reach the sharp regime with respect to the dimension at least when $0 \in \Int (\O )$. 
Note that in this case we have $\langle x,\eta (x) \rangle \geq 0$ for all $x\in \dO$ so that the assumption $(A_2)$ of Theorem \ref{theo:main} is trivially satisfied for any non-decreasing radial function $w$. 
We now take $W (x) = w (r) \, I$ with the non-decreasing radial function $w(r) =  \exp \, (\varepsilon r^\alpha /\alpha )$ with $ \varepsilon >0$ to be chosen conveniently. In contrast to the previous situation, observe that the sign in front of $\ve$ is now a $+$, which makes a great difference for our purpose. For $x\in \Omega$, we have
\begin{align*}
\rho \left( \nabla^2 V (x) - \L W (x) \, W^{-1} (x)\right) & = (\alpha -1 - \varepsilon \, (d+ \alpha-2) ) \, r^{\alpha-2} +\varepsilon (1-\varepsilon ) r^{2(\alpha-1)} \\
& \geq (\alpha -1 - \varepsilon \, (d+ \alpha-2) ) \, r^{\alpha-2} + \frac{\varepsilon}{2} \, r^{2(\alpha-1)},
\end{align*}
which is positive
as soon as $\varepsilon \in  (0,(\alpha-1)/(d+\alpha-2)) \subset (0,1/2]$. Denoting $\varphi$ the latter function of $r \in (0,R_\tmax ]$, one observes that the minimum of $\varphi$ on $\R^+$ is attained at point
$$
r_0 = \left( \frac{(2-\alpha) (\alpha-1 - \varepsilon (d+\alpha-2))}{ \varepsilon (\alpha-1)} \right) ^{1/\alpha} ,
$$
so that for all $r\in (0,R_\tmax]$,
$$
\varphi (r) \geq \frac{\alpha \, (\alpha-1 - \varepsilon (d+\alpha-2))}{2(\alpha-1)} \, \left( \frac{(2-\alpha)(\alpha-1 - \varepsilon (d+\alpha-2))}{\varepsilon (\alpha-1)}\right) ^{1-2/\alpha} .
$$
Choosing the parameter
$$
\varepsilon = \frac{\alpha-1}{2(d+\alpha-2)} \, \in (0,1/2),
$$
yields the inequality
$$
\rho \left( \nabla^2 V (x) - \L W (x) \, W^{-1} (x)\right) \geq \frac{\alpha }{4} \, \left( \frac{2-\alpha}{\alpha-1}\right) ^{1-2/\alpha} \, (d+\alpha-2) ^{1-2/\alpha} ,
$$
which ends the  proof in the case $1<\alpha <2$. 
\end{proof}

Before commenting the results obtained in Proposition \ref{prop:subbotin}, let us mention that similarly to Corollary \ref{corol:main2bis} we are also able to state a generalized Brascamp-Lieb inequality in the Subbotin case. For instance when $\alpha \geq 2$, the above choice $w(r) =  \exp \, (- \varepsilon r^\alpha /\alpha )$ with $\ve = \beta_\alpha / R^{\alpha} _{\max}$ and $\beta_\alpha = \min \{ \beta,1\}$ ensures on the one hand that the assumption $(A_2)$ of Theorem \ref{theo:main} is satisfied, and on the other hand  
\begin{align*}
\rho \left( \nabla^2 V (x) - \L W (x) \, W^{-1} (x)\right) 
& = (1 - \ve  r^{\alpha}) r^{\alpha-2} + \ve (d +\alpha- 2 - \ve r^{\alpha}) r^{\alpha-2} 
\\
& \geq (1-\beta_\alpha) r^{\alpha-2}+ \frac{\beta_\alpha  (d+\alpha- 2-\beta_\alpha )}{R_{\tmax}^\alpha} \, r^{\alpha-2}\\
& \geq \frac{1-\beta_\alpha}{R_\tmax ^\alpha} \, r^{2(\alpha-1)}+ \frac{\beta _\alpha (d+\alpha- 2-\beta_\alpha )}{R_{\tmax}^\alpha} \, r^{\alpha-2},
\end{align*}
so that it leads to the following inequality: for all $g\in \C ^\infty (\Omega)$, 
\[
\Var_\mu(f) \leq \int_\O |\nabla f|^2  \frac{r^2} {\beta_\alpha (d+\alpha -2-\beta_\alpha ) + (1-\beta_\alpha )r^\alpha}  \left( \frac{R_\tmax}{r} \right)^\alpha d\mu.
\]
On the other hand when $0 \in \Int (\O )$, choosing $w(r) =  \exp \, ( \varepsilon r^\alpha /\alpha )$ with $\ve >0$ some conveniently chosen parameter (assumption $(A_2)$ of Theorem \ref{theo:main} is then automatically satisfied) entails that
\begin{align*}
\rho \left( \nabla^2 V (x) - \L W (x) \, W^{-1} (x)\right) 
& = (1 - \varepsilon \, (d+ \alpha-2) ) \, r^{\alpha-2} +\varepsilon (1-\varepsilon ) r^{2(\alpha-1)} \\
& \geq C_{\alpha,d} \, \left( r^{\alpha-2} + r^{2(\alpha-1)} \right) ,
\end{align*}
with $C_{\alpha ,d}$ some explicit constant depending on $\alpha$ and behaving as $1/d$ as $d$ goes to infinity, so that it leads to another generalized Brascamp-Lieb inequality similar to the one we obtained previously on the whole space $\R^d$ (cf. page 1049 in \cite{ABJ}). As such, it is not clear how to compare these two inequalities between them and with \eqref{eq:GBL_r2}, even in the Gaussian case $\alpha=2$. \smallskip 

Now let us come back to the spectral results obtained in Proposition \ref{prop:subbotin}. Actually, the various lower bounds on the spectral gaps appear with a maximum, meaning that some comparison depending on the parameters of interest has to be done. More precisely, it is interesting to note that one has to compare $R_\tmax^\alpha$ and the dimension $d$ to find which estimate is the best, justifying the discussion on the ball $\B (0,R)$ before Proposition \ref{prop:subbotin}. Moreover, to see the relevance of our estimate, let us focus on the ball $\B (0,R)$ for which some results are already available in the literature. According to \cite{bobkov_radial,bjm}, this radial situation can be reduced to a careful study of the one-dimensional radial part and in this case we have the two-sided estimates, cf. \cite{bjm}:
$$
\frac{(d-1) \, \mu (\B (0,R))}{\int_{\B (0,R)} \vert x\vert ^2 \, d\mu (x)} \leq \lambda_1 (\B (0,R ) ,\mu ) \leq \frac{d \, \mu (\B (0,R))}{\int_{\B (0,R)} \vert x\vert ^2 \, d\mu (x)} .
$$
Hence the following asymptotic result holds:
$$
\lambda_1 (\B (0,R) ,\mu) \underset{d \to \infty}{\sim} \frac{d \, \mu (\B (0,R))}{ \int_{\B (0,R)} \vert x\vert ^2 \, d\mu (x)}.
$$
Passing then in polar coordinates and using Laplace's method for the estimation of integrals leads to the exact asymptotics
$$
\lambda_1 (\B (0,R) ,\mu) \underset{d \to \infty}{\sim} \max \left \{ \frac{d}{R^2}, d^{1-2/\alpha} \right \} .
$$
Hence one deduces that the estimates of Proposition \ref{prop:subbotin} applied to the ball $\B (0,R)$ directly recover these asymptotics up to some unessential prefactors depending only on $\alpha$ since we have $\beta = 1$ and $R_\tmax  = R$. \smallskip 

As mentioned just after the proof of Corollary \ref{corol:main2bis}, there exists a situation for which Payne-Weinberger's result is better than ours: when the diameter of the convex body is sufficiently small compared to $R_{\max}$ (enforcing the distance $d(0,\O)$ to be large). Note that since we always have $$
d(0, \O) \leq R_{\max} \leq d(0,\O) + \diam (\O ), 
$$ 
$R_{\max}$ and $d(0,\O)$ are of the same order as soon as the diameter is sufficiently small compared to $R_{\max}$. Therefore Payne-Weinberger's estimate \eqref{eq:Payne_Wein} is better than ours appearing in Proposition \ref{prop:subbotin} in the range 
$$
\frac{1}{\diam (\O ) ^{2}} \, \gg \, \max \left \{ \frac{d}{R_{\max} ^2} , R_{\max} ^{\, \alpha -2} \right \} .
$$
(in the maximum above we ignore the unessential prefactors which do not depend on the quantities of interest, $R_{\max}$, the dimension and the diameter). However using the rotational invariance of the Subbotin distribution will allow us to improve our estimates in this setting by considering functions that depend only on $d-1$ coordinates. This idea already appeared in \cite{cat_guil_mic} for the Gaussian case and revealed to be fruitful to estimate the spectral gap. Assume that $\O$ is far from the origin and that its diameter is sufficiently small. Since the Subbotin measure is radial, one can make a change of coordinates by a rotation and assume that $\Int (\O ) \cap \R e_1 \neq \emptyset$, where recall that $e_1$ is the first vector of the standard canonical basis of $\R^d$. Given some vector $x\in \R^d$, we set $x_{-1} = (0,x_2,\dots,x_d)$ and $r_{-1} = \vert x_{-1} \vert $. Note that one always has $r\geq r_{-1}$. We denote
$$
R_{-1,\tmax} = \underset{ x\in \Omega }{\sup} \, r_{-1} , 
$$
which is always smaller than (and sometimes comparable to) the diameter. The assumption in force now, which is the analogue of \eqref{eq:beta_plus} adapted to the present situation, is the following geometric condition: there exists $\tilde{\beta} \geq 0$ such that 
\begin{equation}
\label{eq:beta-1+}
\J\eta (x) \geq \tilde{\beta}  \, \frac{\langle x_{-1}, \eta (x) \rangle ^+ }{r_{-1}^2} \, I , \quad x\in \dO .
\end{equation}
Although the meaning of this geometric condition is not completely clear, it is satisfied with $\tilde{\beta} = 1$ for the ball $\B(a,R)$. Indeed we have 
$\J\eta (x) = (1/ R) \, I$ for all $x\in \dO$ and since $x= a + R\eta(x)$, we get $x_{-1}= R \, \eta (x) _{-1}$ and 
\[
\frac{\langle x_{-1}, \eta (x) \rangle ^+ }{r_{-1}^2} = \frac{R \vert \eta(x)_{-1} \vert^2}{R^2 \vert \eta(x)_{-1} \vert ^2} = \frac{1}{R}.
\]
Note that if there exists some $x \in \dO$ such that $r_{-1} = 0$, then a necessary condition to get \eqref{eq:beta-1+} at point $x$ is $\langle x_{-1}, \eta (x) \rangle \leq 0$. \smallskip 

Now we can state the desired improvement of Proposition \ref{prop:subbotin} when considering the above situation. As we will see in the proof, our idea is to consider $W = w I$ with $w$ a function depending only on $r_{-1}$ instead of $r$ as in the proof of Proposition \ref{prop:subbotin}. The price to pay for such a modification is not too high since we only lose a bit on the dimension in the constants ($d$ is replaced by $d-1$).
\begin{prop}
\label{prop:subbotin-non-centree}
Let $\O \subset \R^d$ ($d>4$) be a smooth convex body such that $\Int (\O ) \cap \R e_1 \neq \emptyset$. Assume that $\O$ satisfies \eqref{eq:beta-1+} for some $\tilde{\beta} \geq 0$. Let $\mu$ be the Subbotin distribution with parameter $\alpha >1 $ on $\O$. Denote $\tilde{\beta}_\alpha = \min \{ \tilde{\beta} ,1, \alpha-1 \}$. Then the spectral gap satisfies: 
\begin{itemize}
\item If $\alpha >2$, then 
\[
\lambda_1 \left( \O , \mu \right) \geq  \max \left \{  \frac{ \tilde{\beta}_\alpha (d-3-\tilde{\beta}_\alpha )}{R^2_{-1,\max}}, \, C_{\alpha} \tilde{\beta}_\alpha \, \left(d- 3 -\frac{\tilde{\beta}_\alpha }{2}\right)^{1- 2/\alpha}, d(0,\Omega)^{\alpha-2} \right \} ,
\]
with
$
C_{\alpha} = \frac{\alpha}{4} \left( \frac{\alpha-2}{2}  \right)^{(2-\alpha)/\alpha} $.
\item If $\alpha \in (1,2]$, then  
\[
\lambda_1 \left( \O , \mu \right) \geq \max \left \{ \frac{ \tilde{\beta}_\alpha (d +\alpha -3- \tilde{\beta}_\alpha )}{R^2_{-1,\max}} ,  (\alpha-1) \, R_{\max} ^{\, \alpha-2} \right \}.
\]
\end{itemize}
\end{prop}

\begin{proof} 
The proof is somewhat similar to that of Proposition \ref{prop:subbotin} except we consider now $W = w I$ with $w$ a function depending only on $r_{-1}$. In the case $\alpha >2$ we set $w(r_{-1}) = r_{-1} ^{-\tilde{\beta}_\alpha}$ so that on the one hand the hypothesis \eqref{eq:beta-1+} entails that the assumption $(A_2)$ of Theorem \ref{theo:main} is satisfied and on the other hand, 
the smallest eigenvalue of the matrix $\nabla^2 V (x) - \L W (x) \, W^{-1} (x)$ for $x\in \O$ is given by
\begin{align*}
\label{eq:subb:1-r}
\rho \left( \nabla^2 V (x) - \L W (x) \, W^{-1} (x)\right) & =  (1- \tilde{\beta}_\alpha )   \, r^{\alpha-2} + \frac{\tilde{\beta}_\alpha (d-3-\tilde{\beta}_\alpha )} {r_{-1} ^2} \\
& \geq (1- \tilde{\beta}_\alpha ) \, r_{-1} ^{\alpha-2} + \frac{\tilde{\beta}_\alpha (d-3-\tilde{\beta}_\alpha )} {r_{-1} ^2},
\end{align*} 
because $r\geq r_{-1}$ and $\alpha >2$. Since the optimization procedure is then exactly the same as in Proposition \ref{prop:subbotin}, the case $\alpha >2$ is achieved. \smallskip 

\noindent As in Proposition \ref{prop:subbotin}, to improve a bit the constant with respect to the dimension in the case $1<\alpha\leq 2$, we rather consider the function $w(r_{-1})=\exp(-  \ve r_{-1}^\alpha /\alpha)$ with $\ve = \tilde{\beta}_\alpha / R_{-1,\max} ^{\alpha}$. Hence the assumption $(A_2)$ of Theorem \ref{theo:main} is satisfied and for all $x\in \O$ we have 
\begin{align*} 
\rho \left( \nabla^2 V (x) - \L W (x) \, W^{-1} (x)\right) & = (\alpha -1) r^{\alpha-2}  + \ve (d+\alpha-3) r_{-1}^{\alpha-2} - \ve r^{\alpha-2} \, r_{-1}^\alpha - \ve^2  r_{-1}^{2(\alpha-1)} \\
& = (\alpha-1 - \ve  r_{-1} ^{\alpha}) r^{\alpha-2} + \ve (d +\alpha- 3 - \ve r_{-1} ^{\alpha}) r_{-1} ^{\alpha-2} 
\\
& \geq \frac{\tilde{\beta}_\alpha (d+\alpha-3-\tilde{\beta}_\alpha )}{R_{-1,\tmax} ^2}.
\end{align*} 
The proof is now complete.
\end{proof}


As mentioned previously, the assumption \eqref{eq:beta-1+} is the direct analogue of \eqref{eq:beta_plus} adapted to our situation, but it is not completely satisfactory since its geometrical meaning seems difficult to understand. Hence a stronger but more readable assumption relies on the uniform convexity of the convex body.

 
\begin{prop}\label{prop:Subbotin-petites-boules-non-centrees}
Let $\O \subset \R^d$ ($d \geq 2$) be a smooth uniformly convex body such that $\Int (\O ) \cap \R e_1 \neq \emptyset$. Denote $\rho = \inf _{\dO} \rho (\J \eta) >0$. Let $\mu$ be the Subbotin distribution with parameter $\alpha >1 $ on $\O$. Then we have the spectral gap estimate 
\[
\lambda_1 (\Omega , \mu ) \geq \min \{ c_\alpha , \rho \, R_{-1,\max} \} \, \frac{2(d-1)  }{ 3 R_{-1,\max}^2 },
\]
where $c_\alpha = \min \{ 1,  \alpha-1 \}$.
\end{prop}
 
\begin{proof}
As in the proof of Proposition \ref{prop:subbotin-non-centree}, we consider $W = w I$ with $w$ depending on $r_{-1}$. Let us choose $w$ of the form $w(r_1) = C - r_{-1}^2$ where $C > R_{-1,\tmax}^2 $ is some constant to be chosen later. Given $x\in \dO$, we have
\begin{align*}
\J \eta (x) - W (x)  \, \langle \nabla W ^{-1} (x) , \eta (x) \rangle 
& = \J\eta (x)  -\frac{2 r_{-1} }{C-r_{-1}^2} \, \frac{\langle x_{-1}, \eta (x) \rangle}{r_{-1}}  \, I .
\end{align*}
Since we always have $\langle x_{-1} , \eta (x) \rangle \leq r_{-1}$, the assumption $(A_2)$ of Theorem \ref{theo:main} is satisfied as soon as 
\[
\rho \geq \sup_{x\in \dO } \frac{2 r_{-1} }{C-r_{-1}^2} = \frac{2 R_{-1,\tmax} }{C-R_{-1,\tmax}^2},
\]
that is,  
\[
C\geq R_{-1,\tmax}^2 \, \left( 1 + \frac{2 }{\rho \, R_{-1,\tmax}} \right).
\]
Now we have for all $x\in \O$, 
\begin{align*}
\rho \left( \nabla^2 V (x) - \L W (x) \, W^{-1} (x)\right) 
& = c_\alpha r^{\alpha-2} - r^{\alpha-2} \, \frac{2 r_{-1}^2 }{C-r_{-1}^2}  + \frac{2(d-1)}{C-r_{-1}^2} \\
& = \frac{2(d-1)}{C-r_{-1}^2}  +  \frac{ c_\alpha \, C - (c_\alpha +2) r_{-1}^2} {C-r_{-1}^2} \, r^{\alpha-2}\\
&\geq  \frac{2(d-1)}{C-r_{-1}^2} ,
\end{align*}
provided 
\[
C\geq R_{-1,\tmax}^2 \, \left( 1 + \frac{2}{c_\alpha} \right).
\]
Therefore choosing $C$ as the best constant satisfying both constraints, \textit{i.e.}, 
\begin{align*}
C & = R_{-1,\tmax}^2 \, \left( 1 + \frac{2}{ \min \{ c_\alpha , \rho \, R_{-1,\tmax} \} } \right) ,
\end{align*}
entails  
\begin{align*}
\lambda_1(\Omega , \mu) & \geq \inf_{x\in \O} \frac{2(d-1) }{C-r_{-1} ^2 } = \frac{2(d-1) }{C} ,
\end{align*}
which in turn implies the desired spectral gap estimate since the minimum above involving $c_\alpha$ is smaller than 1. 
\end{proof}

\subsection{The case of generalized Orlicz balls}
\label{sect:orlicz}
To finish this work, let us investigate a somewhat different situation, that is, when the second fundamental form $\J \eta $ is diagonal. This is the case when the function $F$ describing the domain $\O$ has an additive form, that is for instance,
$$
F (x) = \sum_{i=1} ^d U_i (x_i) - 1, \quad x\in \R^d ,
$$
where the potentials $U_i : \R \to \R ^+$ are smooth one-dimensional functions, since we have at the boundary $x\in \dO$,
$$
\eta (x) = \frac{1}{\sqrt{\sum_{i=1} ^d U_i ' (x_i) ^2}} \, \left( U_1 ' (x_1), \ldots, U_d ' (x_d)\right) ^T ,
$$
and
$$
\J \eta (x) = \frac{1}{\sqrt{\sum_{i=1} ^d U_i ' (x_i) ^2}} \, \diag \, U_i '' (x_i), \quad x\in \R^d.
$$
Above the diagonal matrix $\diag \, U_i '' (x_i) $ has the $U_i '' (x_i)$ on the diagonal. In particular when the $U_i$ are convex functions, the domain $\O$ is convex and called a generalized Orlicz ball, cf. \cite{koles_milman_Orlicz} (note that similarly to \cite{koles_milman_Orlicz} we do not assume any symmetry assumption on the $U_i$). Although the forthcoming result might be adapted to general probability measures, in particular product measures on $\O$, let us provide a simplified version in the context of the uniform probability measure on the convex body $\O$. This corollary,  derived from Theorem \ref{theo:main} and which exhibits a dimension free spectral gap estimate, corresponds to the third and last main result of the paper.
\begin{corol}
\label{corol:main3}
Let $\Omega$ be of the form
$$
\Omega = \left \{ x\in \R ^d : \sum_{i=1} ^d U_i (x_i) \leq  1 \right \} ,
$$
where the smooth functions $U_i : \R \to \R^+ $ are convex. We assume moreover the following properties: there exists some $R>0$ such that

$\circ$ $\Omega \subset [-R,R]^d$;

$\circ$ there exists some $q>0$ such that for all $i \in \{ 1,\ldots, d \}$,
$$
q \, \vert U_i ' (x_i) \vert \leq U_i '' (x_i), \quad x_i \in [-R,R].
$$

\noindent Then the spectral gap satisfies
$$
\lambda_1 ( \O ) \geq \frac{1}{R^2} \, \arctan \left( \frac{2Rq}{\pi} \right) ^2.
$$
\end{corol}
\begin{proof}
In contrast to all the previous cases met in this paper, we choose for any $x\in \O$ the diagonal matrix weight $W (x)$ in Theorem \ref{theo:main} of the form $\diag \, w_i (x_i)$, where the $w_i$ are some smooth positive functions on $[-R,R]$. In other words $W$ is still a smooth invertible diagonal matrix mapping, but not necessarily a multiple of the identity. Then we have
$$
\nabla^2 V (x) - \L W(x) \, W^{-1} (x) = \diag \, \frac{- w_i '' (x_i)}{w_i (x_i)} , \quad x\in \O ,
$$
and at the boundary $x\in \dO$,
$$
\J \eta (x) - W(x)  \, \langle \nabla W^{-1} (x), \eta (x)\rangle = \frac{1}{\sqrt{\sum_{i=1} ^d U_i '(x_i)^2}} \, \diag \, \left( U_i '' + \frac{w_i '}{w_i} \, U_i ' \right) (x_i).
$$
For any $i \in \{ 1,\ldots, d\}$ we set $w_i (x_i) = \cos (\varepsilon x_i)$ for some relevant $\varepsilon >0$ to be determined thereafter (depending on the parameters $q$ and $R$). Hence the assumption $(A_2)$ in Theorem \ref{theo:main} is satisfied as soon as
$$
U_i ''(x_i) \geq \varepsilon \, \tan (\varepsilon x_i) \, U_i '(x_i) , \quad x_i \in [-R,R].
$$
Then the choice $\varepsilon = \arctan (2Rq /\pi)/R \subset (0,\pi /2R)$ guarantees this inequality (and the fact that the $w_i$ are positive). Finally, the assumption $(A_1)$ in Theorem \ref{theo:main} holds and we obtain the spectral gap estimate
$$
\lambda_1 (\O) \geq \inf_{x\in \O} \min_{i=1,\ldots,d} \frac{- w_i '' (x_i)}{w_i (x_i)} = \varepsilon ^2 ,
$$
which is the desired result.
\end{proof}

It is worth noticing that the dimension free estimate of Corollary \ref{corol:main3} seems useless when the spectral gap is expected to depend on the dimension. For instance in the case of the $\ell ^p$ unit ball with $p \geq 1$, denoted $\B_p$, the potentials $U_i$ are of the form $U_i (x_i) = \vert x_i \vert ^p$, $x_i\in [-1,1]$, and the spectral gap $\lambda_1 (\B_p )$ is of order $d^{2/p}$, cf. \cite{sodin} for the case $p \in [1,2]$ and \cite{latala} for $p \geq 2$. Thus it satisfies the famous KLS conjecture. However our estimate becomes relevant as $p$ tends to infinity since Corollary \ref{corol:main3} entails, when applied to the $\ell ^p$ unit ball with $q=p-1$ and $R = 1$, the lower bound $\arctan \left( 2(p-1)/\pi \right) ^2$ which converges to $\pi ^2 /4$. This quantity is the expected value of the spectral gap obtained by tensorization, the $\ell^\infty$ unit ball being nothing but the hypercube $[-1,1]^d$. \smallskip

To go further into the analysis, it is known that there is no monotonicity properties of the spectral gap with respect to the inclusion of domains, even in the convex case. Indeed one could believe \textit{a priori} that, similarly to the one-dimensional case, the spectral gap decreases when the convex domain increases since it is intimately related to the speed of convergence to equilibrium of the underlying Brownian motion. Nevertheless, considering some thin rectangle $\O \subset [-R,R]^d$ localized around the diagonal of the hypercube shows that this intuition is false: since its largest side is of order $R \sqrt{d}$, the spectral gap $\lambda_1 (\O)$ is of order $1/d R^2$ whereas $\lambda_1 ([-R,R]^d ) = \pi ^2 /4R^2$. Note however that Klartag \cite{klartag} proved a kind of monotonicity property in the unconditional situation: if $\O \subset [-R,R] ^d$ is unconditional then
$$
\lambda_1 \left( \O \right) \geq \lambda_1 ( [-R,R]^d ).
$$
Hence Corollary \ref{corol:main3} can be seen as a generalization of Klartag's result beyond the unconditional setting as soon as the parameter $q$ does not depend on the dimension. For instance it should be applied to some domain $\O$ involving non symmetric one-dimensional potentials on a centered bounded interval of the type
$$
U_i (x_i) = 1_{\{ x_i \geq 0\}} \, \vert x_i \vert ^{p_i} + 1_{\{ x_i < 0\}} \, \vert x_i \vert ^{q_i},
$$
for $p_i ,q_i \geq 1$. See also the work of Kolesnikov and Milman \cite{koles_milman_Orlicz} in which the authors show that the generalized Orlicz balls $\O_E = \left \{ x\in \R ^d : \sum_{i=1} ^d U_i (x_i) \leq  E \right \} $ (without the boundedness restriction $\O_E \subset [-R,R]^d$ for some $R>0$) satisfy the KLS conjecture for certain levels $E \in \R$ under an assumption on the rate of growth at infinity of the $U_i$. \smallskip

\section{Appendix}

Considering the case of a product measure on a cube, we are able to reach optimality in Theorem \ref{theo:main}. 
Indeed the idea is to take the matrix mapping $W$ as the Jacobian matrix of the diffeomorphism whose coordinates are the eigenfunctions related to the spectral gap of the one-dimensional marginal distributions. As such, the weight $W$ satifies the assumptions $(A_1)$ and $(A_2)$ and is thus diagonal (for assumption $(A_2)$, an approximation procedure somewhat similar to that emphasized in Corollary \ref{corol:main3} is required). In this appendix, we focus our attention on the spectral gap $\lambda_1 (\B (0,R))$ of the centered Euclidean ball $\B (0,R)$ of radius $R>0$ endowed with the uniform distribution and wonder if a possible optimality in Theorem \ref{theo:main} can be reached in this non-product context. We will see that such an analysis leads to an interesting phenomenon and opens the door to a natural open question. Before going further into the details, let us recall how to identify the exact expression of the spectral gap $\lambda_1 (\B (0,R))$, cf. for instance Weinberger \cite{weinberger}. Actually, the associated eigenspace is known to be of dimension $d$ and the corresponding eigenfunctions are given (in a vector field notation) by
$$
F(x) = \frac{g(r)}{r} \, x , \quad x\in \B (0,R),
$$
where $g$ is solution to the equation
$$
g''(r) + (d-1) \, \frac{g'(r)}{r} + \left( \lambda_1 (\B (0,R)) - \frac{d-1}{r^2}\right) \, g = 0,
$$
which vanishes at $0$. This is a generalized Bessel equation and classical computations provide the generic solution given by
\begin{equation}
\label{eq:g-bessel}
g(r) = \left( \sqrt{\lambda_1 (\B (0,R))} \, r \right) ^{1 -\frac{d}{2}} \, J_{\frac{d}{2}} \left(\sqrt{\lambda_1 (\B (0,R))} \, r \right),
\end{equation}
where $J_{d/2}$ stands for the Bessel function of the first kind $J_\varsigma$, \textit{i.e.},
$$
J_\varsigma (r) := \sum_{k=0} ^\infty \frac{(-1)^k \left(\frac{r}{2} \right)^{\varsigma +2k}}{k! \Gamma (\varsigma + k +1)},
$$
with $\varsigma = d/2$. Moreover a standard analysis shows that the ratio $g'/g$ is non-negative up to the first zero of $g'$. We have
\begin{equation}
\label{eq:jac}
\Jac \, F (x) = \frac{g(r)}{r} \, I + \left( g'(r) - \frac{g(r)}{r} \right) \, \frac{x x^T}{r^2}  ,
\end{equation}
and since we have $\eta (x) = x/r$, we get at the boundary $\S (0,R)$ (the sphere centered at the origin and of radius $R$),
$$
\Jac \, F (x) \, \eta (x) = \frac{g'(R)}{R} \, x ,
$$
so that each coordinate of the vector field $F$ satisfies the Neumann boundary conditions if and only if $g'(R) = 0$. In other words, if we note the function $\tilde{J}_\varsigma : u \mapsto u^{1-d/2} J_{d/2}(u) $, it means that the derivative at point $u = \sqrt{\lambda_1 (\B (0,R))} \, R$ of $\tilde{J}_\varsigma $ applied with $\varsigma = d/2$ vanishes. Thus if $p_{\varsigma}$ denotes the first positive zero of $\tilde{J}_\varsigma$, then the spectral gap $\lambda_1 (\B (0,R))$, corresponding to the smallest positive eigenvalue, is given by
\begin{equation}
\label{eq:trou_boule}
\lambda_1 ( \B (0,R)) = \frac{p_{\frac{d}{2}} ^2}{R^2}.
\end{equation}

\noindent Actually, an interesting question is the following: which radial function leads to the largest spectral gap estimate ? After some computations somewhat similar to the previous ones, such a radial function is given by
$$
w(r) = \left( \sqrt \lambda r \right) ^{1 -\frac{d}{2}} \, J_{\frac{d}{2} - 1} \left(\sqrt \lambda r \right),
$$
for some convenient $\lambda >0$ determined when saturating the boundary condition $w(R) + R \, w'(R) \geq 0$, that is, $\lambda = p_{d/2 \, -1} ^2 /R^2$, so that we obtain
$$
\lambda_1 ( \B (0,R)) \geq \frac{p_{\frac{d}{2} -1} ^2}{R^2} .
$$
However this lower bound is not really explicit in terms of the dimension, as the optimal one \eqref{eq:trou_boule}. \smallskip

Let us return to our original questioning about a potential optimality in Theorem \ref{theo:main} in the non-product context of the uniform measure on $\B (0,R)$. Similarly to the product measure case, we intend to choose the weight $W$ as the Jacobian matrix of the eigenfunctions associated to the spectral gap $\lambda_1 (\B (0,R))$, cf. the formula \eqref{eq:jac}. Although $W$ is not diagonal (thus Theorem \ref{theo:main} cannot be used directly), let us observe however how the assumptions $(A_1)$ and $(A_2)$ are satisfied. Since $W$ is not invertible on the boundary $\S (0,R)$, we still work on the ball $\B (0,R)$ but we consider on the extended ball $\B(0,R+\ve)$ for some $\ve>0$ the smooth matrix mapping $W_\ve = \Jac \, F_\ve$. Here $F_\varepsilon$ denotes the vector field whose coordinates are the eigenfunctions associated to the spectral gap $\lambda_1 (\B (0,R+ \ve ))$, with the corresponding function $g_\ve$ defined analogously to \eqref{eq:g-bessel}. From the computations below, we will see that $W_\ve$ is invertible on $\B (0,R)$. We simply denote $W$ and $g$ the respective quantities for $\ve = 0$. Moreover we (still) denote $\Delta$ the diagonal matrix operator with the Laplacian acting on functions on the diagonal. \\
On the one hand the identity $\Delta F_\ve = - \lambda_1 (\B (0,R+\ve))\, F_\ve$ holds on the smaller ball $\B (0,R)$ and leads to
$$
\nabla^2 V - \L W \, W^{-1} = - \Delta \Jac \, F_\ve \, (\Jac \, F_\ve )^{-1} = \lambda _1(B(0,R+\ve) ) \, I,
$$
since the Laplacian commutes with the Jacobian. Hence assumption $(A_1)$ is satisfied. On the other hand we need some additional computations to verify assumption $(A_2)$. Recall that by \eqref{eq:jac} we have
$$
W_\ve (x) = \frac{g_\ve (r)}{r} \, I  + \left( g_\ve '(r) - \frac{g_\ve (r)}{r} \right) \, \frac{x x^T}{r^2} , \quad x\in \B(0,R).
$$
Above the matrix $x x^T /r^2$ is that of the orthogonal projection onto $\R x$. In particular the matrix $W_\ve (x) $ is diagonalizable and $x$ is an eigenvector associated to the eigenvalue $g_\ve '(r)$ whereas any vector of the orthogonal complement $(\R x )^\perp$ is an eigenvector  associated to the eigenvalue $g_\ve (r)/r$. Since $g_\ve$ and $g' _\ve$ do not vanish on $(0,R+ \ve )$, thus on $(0,R]$, $W_\ve$ is invertible on $\B (0,R)$ and
$$
W_\ve ^{-1} (x) = \frac{r}{g_\ve (r)} \, I + \left(\frac{1}{g_\ve '(r)} - \frac{r}{g_\ve (r)}\right) \, \frac{x x^T}{r^2}, \quad x\in \B (0,R).
$$
Note that because $\eta (x) = x/r$ we have for all $i,j = 1,\ldots, d$,
$$
\left \langle \nabla \left( \frac{x_i x_j}{r^2} \right), \eta (x) \right \rangle = \sum_{k=1} ^d \left( \frac{x_j}{r^2} \, \delta _{i,k} + \frac{x_i}{r^2} \, \delta _{j,k} - \frac{2x_i x_j x_k }{r^4} \right) \, \frac{x_k}{r}  = 0,
$$
where $\delta$ is the usual Kronecker delta symbol. Since for any smooth radial function $h$ we have $\nabla h (r) = h'(r) x/r$, we obtain
$$
\left \langle \nabla W_\ve ^{-1} (x) , \eta (x) \right \rangle = \frac{\partial }{\partial r} \left(\frac{r}{g_\ve (r)}\right) \, I + \frac{\partial }{\partial r} \left(\frac{1}{g_\ve '(r)} - \frac{r}{g_\ve (r)}\right)\,  \frac{x x^T}{r^2}.
$$
Finally at the boundary $x\in \S (0,R)$, we have on the hyperplane $\eta (x) ^\perp$,
\begin{eqnarray*}
\Jac \, \eta (x) - W_\ve (x)  \, \left \langle \nabla W_\ve ^{-1} (x), \eta (x) \right \rangle & = & \left( \frac{1}{R} - \frac{g_\ve (R)}{R} \, \frac{\partial }{\partial r} \left(\frac{r}{g_\ve (r)}\right) {\Big | _ {r = R}} \right) \, I \\
& = & \frac{g' _\ve (R)}{g_\ve (R)} \, I ,
\end{eqnarray*}
meaning that assumption $(A_2)$ is satisfied for each fixed $\ve >0$ since $g' _\ve /g_\ve \geq 0$ on $(0,R+\ve)$ and thus on $(0,R]$. Finally assumption $(A_1)$ is satisfied with the (optimal) constant matrix $\lambda_1 (\B (0, R)) \, I$ by passing to the limit $\ve \to 0$ since $\lambda_1 (\B (0, R+\ve)) \to \lambda_1 (\B (0, R))$. Moreover $g' _\ve (R) /g_\ve (R) \to g' (R) /g (R) = 0$, \textit{i.e.}, the boundary term in $(A_2)$ vanishes. The validity of these limits can be obtained either by a classical continuity argument or by using directly the exact expression \eqref{eq:trou_boule} of the spectral gap of any Euclidean ball and plugging then into the formula \eqref{eq:g-bessel} defining $g_\ve$. \smallskip

As announced earlier, the previous discussion does not allow us to use directly Theorem \ref{theo:main} because the matrix mapping $W$ is not diagonal. However we can apply Lemma \ref{lemme:decompo} with the above weight $W_\ve$ and let $\ve \to 0$ to get the following identity: if $\mu$ is the uniform probability measure on $\B (0, R)$, then for all $f\in \C _N ^\infty (\B (0, R))$,
\begin{eqnarray*}
\nonumber \int_{\B (0, R)} (-Lf)^2 \, d\mu & = & \int_{\B (0, R)} \left[ \nabla (W^{-1} \nabla f) \right] ^T \, W^T W \, \nabla (W^{-1} \nabla f) \, d\mu \\
\nonumber & & + \int_{\B (0, R)} \left \langle W^{-1} \nabla f , ( \nabla W^T \, W - W^T \, \nabla W ) \, \nabla (W^{-1}\nabla f) \right \rangle \, d\mu \\
\nonumber & & + \lambda_1 (\B (0,R)) \, \int_{\B (0, R)} \vert \nabla f \vert ^2 \, d\mu .
\end{eqnarray*}
Although the first term is always non-negative, some computations show that for all $i,j=1,\ldots, d$,
$$
\left( \nabla W^T \, W - W^T \, \nabla W \right) _{i,j}= \frac{1}{r^2} \, \left( g'(r) - \frac{g(r)}{r} \right) ^2 \, \left( x_j e_i - x_i e_j \right) ,
$$
where $(e_k)_{k=1,\ldots,d}$ is the usual canonical basis of $\R^d$. Therefore the matrix of vectors $\nabla W^T \, W - W^T \, \nabla W $ is not zero and thus it is not clear to us that the second integral above vanishes, as it is trivially the case when $W$ is diagonal (or when $f$ is one of the eigenfunctions associated to the spectral gap $\lambda_1 (\B(0,R))$ since in this case $W^{-1} \nabla f $ is constant). Nevertheless according to the (integrated version of) Bakry-Emery criterion, cf. \cite{bakry_emery}, we know \textit{a priori} that the sum of these two integrals is non-negative. Hence a challenging question would be to prove directly that this sum is non-negative for all $f\in \C _N ^\infty (\B (0, R))$.


\begin{thebibliography}{10}

\bibitem{KLS_livre} D. Alonso-Guti\'errez and J. Bastero. \textit{Approaching the Kannan-Lov\'asz-Simonovits and variance conjectures}. Lecture Notes in Math. 2131, Springer, 2015.

\bibitem{ABJ} M. Arnaudon, M. Bonnefont and A. Joulin. Intertwinings and generalized Brascamp-Lieb inequalities. \textit{Rev. Math. Ibero.}, 34:1021-1054, 2018.

\bibitem{bakry_emery} D. Bakry and M. Emery. Diffusions hypercontractives. \textit{S\'eminaire de Probabilit\'es}, XIX, 177-206, Lecture Notes in Math., 1123, Springer, 1985.

\bibitem{barthe_cordero} F. Barthe and D. Cordero-Erausquin. Invariances in variance estimates. \textit{Proc. London Math. Soc.}, 106:33-64, 2013.

\bibitem{bobkov_AOP} S.G. Bobkov, Isoperimetric and analytic inequalities for log-concave probability measures. \textit{Ann. Probab.}, 27:1903-1921, 1999.

\bibitem{bobkov_radial} S.G. Bobkov, Spectral gap and concentration for some spherically symmetric probability measures. \textit{Geometric Aspects of Functional Analysis}, 37-43, Lecture Notes in Math., 1807, Springer, Berlin, 2003.

\bibitem{cat_guil_mic} E. Boissard, P. Cattiaux, A. Guillin and L. Miclo. Ornstein-Uhlenbeck pinball and the Poincar\'e inequality in a punctured domain. \textit{S\'eminaire de Probabilit\'es}, XLIX, 1–55, Lecture Notes in Math., 2215, Springer, 2018.

\bibitem{BJ} M. Bonnefont and A. Joulin. Intertwinings, second-order Brascamp–Lieb inequalities and spectral estimates. \textit{Stud. Math.}, 260:285-316, 2021.

\bibitem{bjm} M. Bonnefont, A. Joulin and Y. Ma. Spectral gap for spherically symmetric log-concave probability measures, and beyond. \textit{J. Funct. Anal.}, 270:2456-2482, 2016.

\bibitem{brascamp_lieb} H.J. Brascamp and E.H. Lieb. On extensions of the {B}runn-{M}inkovski and {P}r\'ekopa-{L}eindler theorems, including inequalities for log-concave functions, and with an application to the diffusion equation. \textit{J. Funct. Anal.}, 22:366-389, 1976.



\bibitem{gromov_milman} M. Gromov and V.D. Milman. Generalization of the spherical isoperimetric inequality to uniformly convex Banach spaces. \textit{ Compos. Math.}, 62:263-282, 1987.

\bibitem{helffer} B. Helffer. Remarks on decay of correlations and {W}itten Laplacians, {B}rascamp-{L}ieb inequalities and semiclassical limit. \textit{J. Funct. Anal.}, 155:571-586, 1998.

\bibitem{hormander} L. H\"ormander. $L^2$ estimate and existence theorems for the $\bar{\partial}$ operator. \textit{Acta Math.}, 113:89-152, 1965.

\bibitem{kls} E. Kannan, L. Lov\'asz, and M. Simonovits. Isoperimetric problems for convex bodies and a localization lemma. \textit{Discrete Comput. Geom.}, 13: 541-559, 1995.

\bibitem{klartag} B. Klartag. A Berry-Essen type inequality for convex bodies with an unconditional basis. \textit{Probab. Theory Related Fields}, 145:1-33, 2009.

\bibitem{klartag_KLS} B. Klartag. Logarithmic bounds for isoperimetry and slices of convex sets. \textit{Ars Inveniendi Analytica}, 2023.

\bibitem{KM} A.V. Kolesnikov and E. Milman. Riemannian metrics on convex sets with applications to {P}oincar\'e and log-{S}obolev inequalities. \textit{Calc. Var. Partial Differ. Equ.}, 55: 1-36, 2016.

\bibitem{milman_koles} A.V. Kolesnikov and E. Milman. Brascamp-Lieb-type inequalities on weighted Riemannian manifolds with boundary. \textit{J. Geom. Anal.}, 27:1680-1702, 2017.

\bibitem{koles_milman_Orlicz} A.V. Kolesnikov and E. Milman. The KLS isoperimetric conjecture for generalized Orlicz balls. \textit{Ann. Probab.}, 46:3578-3615, 2018.


\bibitem{latala} R. Latala and J.O. Wojtaszczyk. On the infimum convolution inequality. \textit{Studia Math.}, 189:147-187, 2008.


\bibitem{milman} E. Milman. On the role of convexity in isoperimetry, spectral gap and concentration. \textit{Invent. Math.}, 177:1-43, 2009.

\bibitem{nazarov} A.I. Nazarov, N.G. Kuznetsov and S.V. Poborchi. V.A. Steklov and the problem of sharp (exact) constants in inequalities of mathematical physics. Preprint 2013. Available at https://arxiv.org/abs/1307.8025.

\bibitem{payne_weinberger} L.E. Payne and H.F. Weinberger. An optimal Poincar\'e inequality for convex domains. \textit{Arch. Rational Mech. Anal.}, 5:286-292, 1960.

\bibitem{barthe_roustant} O. Roustant, F. Barthe and B. Iooss. Poincar\'e inequalities on intervals - application to sensitivity analysis. \textit{Electron. J. Statist.}, 11:3081-3119, 2017.


\bibitem{schneider} R. Schneider. \textit{Convex bodies: the Brunn-Minskowski theory}. Second Expanded Edition. Encyclopedia of Mathematics and Its Applications, 151, Cambridge
University Press, 2014.



\bibitem{sodin} S. Sodin. An isoperimetric inequality on the $\ell^p$ balls. \textit{Ann. Inst. H. Poincar\'e Probab. Statist.}, 44:362-373, 2008.


\bibitem{wang} F.Y. Wang. Modified curvatures on manifolds with boundary and applications. \textit{Potential Anal.}, 41:699-714, 2014.

\bibitem{weinberger} H.F. Weinberger. An isoperimetric inequality for the $N$-dimensional free membrane problem. \textit{J. Rational. Mech. Anal.}, 5:633-636, 1956.

\end{thebibliography}
\end{document}